\documentclass[a4paper,11pt]{amsart}%
\usepackage{hyperref}
\usepackage{pdfsync}
\usepackage{dsfont}
\usepackage{color}
\usepackage{esint}
\usepackage{amsthm}
\usepackage{enumerate}
\usepackage{amsfonts}
\usepackage{amssymb}%
\usepackage{mathtools}
\usepackage{braket}
\usepackage{bm}
\usepackage{amsmath}
\usepackage{graphicx}
\usepackage{caption}
\usepackage{fullpage}
\providecommand{\keywords}[1]{\textbf{\emph{Keywords: }} #1}
\newtheorem{theorem}{Theorem}[section]

\newtheorem{assumption}[theorem]{Assumption}
\newtheorem{corollary}[theorem]{Corollary}

\newtheorem{definition}[theorem]{Definition}
\newtheorem{lemma}[theorem]{Lemma}

\newtheorem{proposition}[theorem]{Proposition}

\theoremstyle{remark}
\newtheorem{remark}[theorem]{Remark}
\newtheorem{example}[theorem]{Example}

\newcommand{\R}{\mathds{R}}
\newcommand{\N}{\mathds{N}}
\newcommand{\Z}{\mathds{Z}}

\newcommand{\T}{\mathbb{T}}

\newcommand{\converges}[1]{ \overset{#1}{\longrightarrow}} 
\newcommand{\M}{\mathcal{M}}

\newcommand{\x}{\mathbf{x}}
\newcommand{\y}{\mathbf{y}}
\newcommand{\z}{\mathbf{z}}
\newcommand{\bH}{\mathbf{H}}

\newcommand{\veps}{\varepsilon}

\DeclareMathOperator{\supp}{supp}

\DeclareMathOperator{\diam}{diam}
\DeclareMathOperator{\diamP}{\diam_\theta(\mathcal{P}(\mathcal{K}_N))} 
\DeclareMathOperator{\Ric}{Ric}

\DeclareMathOperator{\divergence}{div}

\DeclareMathOperator{\Lip}{Lip}

\newcommand{\red}{\color{red}}

\definecolor{mygreen}{rgb}{0.1,0.75,0.2}

\newcommand{\nc}{\normalcolor}

\title{ Gromov-Hausdorff limit of Wasserstein spaces on point clouds  }

\author[N.\ Trillos]{Nicol\'as Garc\'ia Trillos}
\address{Nicol\'as Garc\'ia Trillos, Department of Statistics, University of Wisconsin-Madison. 1300 University Avenue, Madison, WI, USA 53706  \\}
\email{garciatrillo@wisc.edu}

\subjclass[2010]{ 49J45	49J55  49J15  35K05 }

\keywords{discrete optimal transport, random geometric graphs, Gromov-Hausdorff convergence, Wasserstein distance}

\begin{document}
\maketitle

\begin{abstract}
We consider a point cloud $X_n := \{ \x_1, \dots, \x_n \}$ uniformly distributed on the flat torus $\T^d : = \R^d / \Z^d  $, and construct a geometric graph on the cloud by connecting points that are within distance $\veps$ of each other. We let $\mathcal{P}(X_n)$ be the space of probability measures on $X_n$ and endow it with a discrete Wasserstein distance $W_n$ as introduced independently in \cite{Zhou,Maas, Mielke_2011} for general finite Markov chains. We show that as long as $\veps= \veps_n$ decays towards zero slower than an explicit rate depending on the level of uniformity of $X_n$, then the space $(\mathcal{P}(X_n), W_n)$ converges in the Gromov-Hausdorff sense towards the space of probability measures on $\T^d$ endowed with the Wasserstein distance. The analysis presented in this paper is a first step in the study of stability of evolution equations defined over random point clouds as the number of points grows to infinity.
\end{abstract}

\section{Introduction}

Two rapidly expanding research areas in the calculus of variations and PDEs are at the core of this work. On the one hand, the large sample asymptotic analysis of graph based procedures in machine learning. On the other, the analysis of evolution equations on graphs (or discrete Markov chains) which are formulated as gradient flows with respect to discrete transport metrics. The purpose of this paper is to build a bridge between these two areas and to take a first step in the study of large sample stability of several evolution equations defined on random point clouds.

The use of graphs for machine learning has been proposed in the past decades by different authors in connection with several learning tasks. Examples of graph based methods include spectral clustering \cite{vonLux_tutorial}, total variation clustering \cite{szlam2009total}, diffusion maps \cite{COIFMAN20065}, graph Laplacian regularization \cite{smola2003kernels}, trend filtering on graphs \cite{wang2016trend}, and graph $p$-Laplacian regularization \cite{el2016asymptotic}, among others. All of these procedures can be formulated as optimization problems where a graph structure on a data set is explicitly used for constructing intrinsic regularization terms or to enforce geometric constraints. For a general statistical procedure that depends on $n$ input data points, a natural question to ask is whether it is stable or consistent as the number of available data points grows. In the context of graph-based learning, the work \cite{trillos2016continuum} introduced a framework, based on ideas from the calculus of variations, to study the consistency issue for a large family of such optimization problems on graphs, and rigorously connect optimization problems at the discrete level with continuum variational problems. Cheeger and Ratio graph cuts were studied in \cite{ConClus}, and were shown to converge, in the sense of $\Gamma$-convergence, towards analogue partition problems at the continuum level. In \cite{trillos2018variational} a variational approach to study the stability of the so called spectral clustering algorithm was taken. One of the main steps in the analysis involved analyzing the large sample behavior of the spectra of graph Laplacians (for which there are actually rates of convergence; see \cite{Burago,trillos2018spectral}). Other PDE and variational techniques have been used to study optimization problems for supervised and semi-supervised learning tasks (i.e. in the presence of partially labeled data) including the work \cite{slepcev2017analysis} where $p$-Laplacian regularization was studied using variational techniques, and \cite{calder2018game}, where a PDE approach using a viscosity solution interpretation of elliptic PDEs was used to prove similar results. Other works that have used PDE techniques in the context of graph based learning are \cite{calder2017consistency,calder2018properly}.

More or less around the same time, three independent contributions by Chow et al \cite{Zhou}, Maas \cite{Maas}, and Mielke \cite{Mielke_2011}, introduced the notion of discrete transport distance (or discrete Wasserstein space) over a fixed Markov chain (or fixed graph). Following closely the dynamic formulation of optimal transport \`{a} la Benamu-Brenier \cite{Benamou2000}, these works established a close connection between evolution equations on graphs and gradient flows with respect to said discrete transport metrics. In particular, their work provided a characterization for the heat flow in a graph analogous to the one at the continuum level introduced in the seminal work of Jordan, Kinderlehrer, and Otto \cite{OTTO}. Discrete optimal transport distances have been recently used to provide variational formulations to several other discrete evolutionary PDEs \cite{ErbarMaasPorousMediumEqns,Laschos,Maas_2016}, to define Ricci curvature lower bounds \cite{Erbar2012,fathi2016,Mielke2013}, and to define super Ricci flows \cite{KopferErbar}. Geodesics for discrete optimal transport have been studied in \cite{Erbar2018,VillaniBook}.

In this paper we consider the same set-up as in \cite{trillos2016continuum} and restrict our attention to graphs whose sets of nodes are point clouds in Euclidean space randomly sampled from a ground-truth distribution, and whose edges are obtained by connecting nodes with an edge if they are sufficiently close to each other (i.e. random geometric graphs). To be more precise, we think of the data points $\x_1, \dots, \x_n$  as i.i.d. random samples from some distribution $\rho$ supported on some manifold $\M$. A graph on the point cloud is constructed by putting edges between points that are close to each other (closer than some quantity $\veps>0$). The graph captures the similarity among data points and allows one to define the discrete differential operators necessary to make sense of discrete transport metrics. In this geometric graph setting three main questions arise:
\begin{enumerate}
	\item Is the discrete Wasserstein space on a random point cloud an approximation of the Wasserstein space at the continuum level?
	\item Are evolutionary PDEs defined on this discrete Wasserstein space consistent, and in particular converge in the limit towards analogue evolution PDEs on $\M$?
	\item Is it possible to establish consistency of machine learning algorithms (with an evolution flavor) using the mathematics stemming from the answers to the previous questions?
\end{enumerate}
In this paper we give an answer to the first question. We show that in a localized setting (i.e. when $\M$ is the $d$-dimensional torus and $\rho$ is the uniform distribution) under appropriate conditions on $\veps$, discrete transport metrics on random geometric graphs converge, in the Gromov-Hausdorff sense as the number of data points grows, towards the usual Wasserstein space on the space of probability measures supported on $\M$. While here we will not provide an answer to the second or third questions above, we notice that an answer to the first question is actually an important first step for addressing the other two. Indeed, the Gromov-Hausdorff convergence of metric spaces is a good notion of convergence for metric spaces closely connected to the $\Gamma$-convergence framework in \cite{SandierSerfaty} and \cite{Gigli2010} to study stability of gradient flows when these are not necessarily defined on the same metric space.  The ultimate goal of the line of research explored in this paper is to provide concise proofs of large sample limits of evolution equations on graphs, and show that these evolutions converge appropriately to evolution equations on a continuum domain or manifold. From the point of view of data analysis our paper can be conceived as a first step in the attempt to use modern analytical methods to establish consistency results for algorithms whose outputs are not solutions of optimization problems, but rather, the result of an evolution in time process; a few examples are the mean shift algorithm and variants like the blurred mean shift algorithm (see \cite{Carreira}).

The concrete question studied in this paper is related to the works by Gigli and Maas \cite{gigmaas}, and Gladbach, Kopfer, and Maas \cite{MaasKopferGladbach}. In \cite{gigmaas}, the authors study a similar problem to ours, but in the setting where the point clouds are actually aligned on a regular grid, and the associated graphs are nearest neighbor graphs. The work \cite{MaasKopferGladbach} improves in several regards the work \cite{gigmaas}, and investigates necessary and sufficient conditions on more general meshes to guarantee the convergence in the limit towards the Wasserstein space. The general layout of our paper follows closely the one in \cite{gigmaas}, however, our arguments are quite different since in our setting there is little structure in the graph that we can exploit, and for us it will be important to have enough ``averaging" for the homogenization to occur in the limit; in particular, we need the average degree of our graph to grow with $n$ (although at a much smaller rate than $n$, and for instance slightly more than logarithmic growth would suffice). Another aspect of our graphs that will become apparent later on, and which is not present in the graphs considered in \cite{gigmaas,MaasKopferGladbach}, is the fact that geometric graphs have both "non-local" and "local" attributes manifesting at different length scales.

To finish this introduction we would also like to mention the papers \cite{HwaDamelinHero,davis2019}. There, the authors study large sample asymptotic behavior of several metrics defined on random point clouds. One main difference with these metrics and the one that we analyze here is that these are only defined on the point cloud. Moreover, these metrics do not have enough structure to be used to cast evolutionary PDEs on graphs as gradient flows. This contrasts with the properties of discrete optimal transport metrics discussed in \cite{Maas}.

\subsection{Outline}

The rest of the paper is organized as follows. In section \ref{Sec:ContWass} we recall the definition of Wasserstein space in $\T^d$ and revisit its dynamical interpretation \`{a} la Benamu-Brenier. In section \ref{Sec:DisGradients} we describe precisely our discrete set-up and introduce the notions of graph Laplacian, discrete gradient, and discrete divergence.  In section \ref{Sec:DisWass} we define the discrete Wasserstein distance. In section \ref{Sec:Gromov} we revisit the notion of  Gromov-Hausdorff distance of metric spaces. With all the background and assumptions in place we proceed to state explicitly the main results of the paper in section \ref{Sec:main}.  In section \ref{Sec:Apriori} we present some a priori estimates for the discrete Wasserstein distance that will be used in our proofs. In sections \ref{Sec:Smooth} and \ref{Sec:SmoothDisc} we study the smoothening of curves in the continuous and discrete settings. In section \ref{Sec:aux} we collect the last auxiliary results that we need for the proof of our main results. Finally, in section \ref{Sec:Proof} we prove our main results.

\subsection{Wasserstein distance in the continuum and Benamu-Brenier formula}
\label{Sec:ContWass}

Let $\mathcal{P}(\T^d)$ be the set of Borel probability measures on $\T^d$. This  set can be endowed with the Wasserstein distance induced by the Euclidean metric on $\T^d$: for $\mu, \tilde{\mu} \in \mathcal{P}(\T^d)$ the Wasserstein distance $W(\mu, \tilde{\mu})$ is defined as
\[ (W(\mu, \tilde{\mu}))^2 := \min_{\pi \in \Gamma(\mu, \tilde{\mu})} \int_{\T^d \times \T^d}  \lvert x-y \rvert^2  d\pi(x,y),   \] 
where $\Gamma(\mu, \tilde{\mu})$ denotes the set of \textit{couplings} (a.k.a. transportation plans) from $\mu$ into $\tilde{\mu}$, that is, measures on the product space $\T^d \times \T^d$ whose first and second marginals are $\mu$ and  $\tilde{\mu}$ respectively. Here $\lvert x-y \rvert$ represents the (periodic) Euclidean distance (on $\T^d$) between points $x$ and $y$.  


The work \cite{Benamou2000} shows that the Wasserstein space is a length space, and furthermore, that the distance $W(\mu, \tilde{\mu})$ is realized by a curve in $\mathcal{P}(\T^d)$ connecting $\mu$ and $\tilde{\mu}$ that has \textit{minimal action}.  That is, for any $\mu,\tilde{\mu} \in \mathcal{P}(\T^d)$, 
\[  (W(\mu, \tilde{\mu}))^2 = \min \int_{0}^1 \int_{\T^d} \lvert  \vec{V}_t(x)\rvert ^2 d \mu_t(x) dt, \]
where the $\min$ is taken over all solutions $(\mu_t, \vec{V}_t)$ to the continuity equation
\begin{equation}
\frac{d}{dt}\mu_t + \divergence( \vec{V}_t \mu_t   ) =0,  
\label{Conteqn}
\end{equation}
with $\mu_0=\mu$ and $\mu_1=\tilde{\mu}$. In general, \eqref{Conteqn} has to be interpreted in the distributional sense, that is, 
\[ \int_{0}^1 \int_{\T^d} \left( \frac{d}{dt} \phi(x,t) +  \langle \vec{V}_t(x), \nabla_x \phi(x,t) \rangle  \right) d\mu_t(x) dt , \quad \forall \phi \in C_c^\infty(\T^d \times (0,1)), \]
or equivalently 
\[ \frac{d}{dt} \int_{\T^d} \phi(x) d \mu_t(x)  = \int_{\T^d} \nabla \phi \cdot \vec{V}_t(x) d \mu_t(x), \quad \forall \phi \in C^\infty(\T^d).  \]
See Chapter 8 in \cite{GigliBook}.

When the measures $\mu_t$ are absolutely continuous with respect to the Lebesgue measure, i.e. $d\mu_t(x)= \rho_t(x) dx$, and in addition $\rho_t >0$, we may work with \textit{flux vector fields} and the continuity equation \eqref{Conteqn} may be written in \textit{flux form} as
\[  \frac{d}{dt}\rho_t + \divergence( \vec{V}_t  ) =0, \] 
where $\vec{V}_t$ in the above formula is the flux $\rho_t$ times $\vec{V}_t$ in \eqref{Conteqn}. As before we have to interpret this new equation in the distributional sense, but we remark that when the densities and flux vector fields are smooth the equation can be interpreted in the classical sense. Finally, the total action associated to a solution to the continuity equation in flux form is defined as 
\[  \int_{0}^1 \mathcal{A}(\rho_t, \vec{V}_t) dt   \]
where
\[\mathcal{A}(\rho_t, \vec{V}_t):= \int_{\T^d}\frac{\lvert \vec{V}_t(x) \rvert^2}{\rho_t(x)} dx.  \]
With these new definitions, we can now say that for an arbitrary solution to the continuity equation in flux form (connecting the densities $\rho_0$ and $\rho_1$) we have
\[ (W(\mu_0, \mu_1))^2 \leq \int_{0}^1 \mathcal{A}(\rho_t, \vec{V}_t) dt.   \]
In the remainder we use $\vec{V}_t$ to denote both, vector fields and flux vector fields, and no confusion should arise by doing so.

\subsection{Gradients, divergences, and Laplacians on weighted graphs}
\label{Sec:DisGradients}

A weighted graph $(X, \omega)$ consists of a set 
$X$ with $N$ elements, and a non-negative function $\omega: X \times X \rightarrow [0, \infty)$. In the remainder we will always assume that $\omega$ is a symmetric function. 

Let $\gamma$ be a base probability measure on $X$ and let $\veps$ be a positive number. We abuse notation slightly and write $\gamma(\x)$ instead of $\gamma(\{ \x\})$, and assume that $\gamma$ gives positive mass to every point in $X$. Let $L^2(\gamma)$ be the set of functions of the form $u : X \rightarrow \R$. In other words,  $L^2(\gamma)$ is the set of real valued functions on $X$. Naturally, $L^2(\gamma)$ can be endowed with an inner product defined by
\[  \langle u , v \rangle_{L^2(\gamma)}  := \sum_{\x \in X} u(\x)v(\x) \gamma(\x). \]
We denote by $\mathcal{P}(X)$ the set of all probability measures on $X$. From the fact that $\gamma$ gives positive mass to every point in $X$ it follows that there is a one to one correspondence between $\mathcal{P}(X)$ and the set of functions $\rho: X \rightarrow [0,\infty)$ 
for which
\[ \int \rho(x) d \gamma(x) = \sum_{\x \in X} \rho(\x)\gamma(\x) =1.\]
We will refer to these functions $\rho$ as \textit{discrete densities}, and use the terms ``density'' and ``measure'' interchangeably. 

Let $\mathfrak{X}(X)$ be the set of functions $V: X \times X \rightarrow \R $ (i.e. the set of \textit{discrete vector fields}), and let $\langle \cdot , \cdot \rangle_{\mathfrak{X}(X)}$ be the inner product 
\[ \langle U  , V  \rangle_{\mathfrak{X}(X)}  :=  \sum_{\x, \y \in X}  U(\x,\y)V(\x, \y) \omega(\x, \y)  \gamma(\x) \gamma(\y), \quad U , V \in \mathfrak{X}(X).  \]

The \textit{divergence} operator $\divergence_N$ takes discrete vector fields and returns functions on $X$, and the \textit{gradient} operator $\nabla_N$ takes functions on $X$ and returns discrete vector fields. More precisely, we define $\divergence_N : \mathfrak{X}(X) \rightarrow L^2(\gamma)$ , $\nabla_N : L^2(\gamma) \rightarrow \mathfrak{X}(X)$ by 
\[ \divergence_N(V) (\x) := \frac{1}{\veps} \sum_{\y \in X}   \left( V(\x, \y) - V(\y, \x)   \right)\omega(x,y) \gamma(\y) , \quad V \in \mathfrak{X}(X),   \]
\[ \nabla_N u (\x, \y) := \frac{u(\y) - u(\x)}{\veps}, \quad u \in L^2(\gamma).   \]

From the symmetry of the weight function $\omega$ it follows, after a straightforward computation, that $\divergence_N$ and $-\nabla_N$ are adjoint operators. 
\begin{proposition}
Suppose that $\omega(\x, \y)= \omega(\y, \x)$ for every $\x, \y \in X$. Then, for any $u \in L^2(\gamma)$ and $U \in \mathfrak{X}(X)$ we have
\[ \langle \divergence_N(U) , u \rangle_{L^2(\gamma)} = - \langle U, \nabla_N u \rangle_{\mathfrak{X}(X)}. \]
\label{divgrad} 
\end{proposition}

The \textit{graph Laplacian} $\Delta_N: L^2(\gamma) \rightarrow L^2(\gamma)$ is the operator defined by
\[ \Delta_n := - \divergence_N \circ \nabla _N.    \]
From Proposition \ref{divgrad} it follow that $\Delta_n$ is a positive semi-definite operator. 

\nc

\subsubsection{The geometric graph setting} Let us now make the previous definitions specific to a geometric graph setting.  We consider a point cloud $X_n:= \{\x_1, \dots, \x_n \}$  contained in $\T^d$ and denote by $\nu_n$ its empirical measure, i.e.
\[ \nu_n := \frac{1}{n} \sum_{i=1}^n \delta_{\x_i}.\]
In this setting, $X_n$ plays the role of $X$ and the empirical measure $\nu_n$ the role of the measure $\gamma$.  We assume that the point cloud $X_n$, or better yet its empirical measure $\nu_n$, approximates the measure $\nu$ (the uniform measure on $\T^d$) in a sense that we now make precise. Given two arbitrary probability measures $\mu, \tilde{\mu} \in \mathcal{P}(\T^d)$, their $\infty$-OT distance, denoted $d_\infty(\mu, \tilde{\mu})$, is defined as
\[ d_\infty(\mu, \tilde{\mu}):= \min_{\pi \in \Gamma(\mu, \tilde{\mu})} \sup_{(x,y) \in \sup(\pi)} \lvert x-y \rvert. \]
The $\infty$-OT distance measures the min-max cost of transporting a given distribution of mass into another. We use this distance to measure how close $\nu_n$ is from $\nu$. We point out that $d_\infty(\nu, \nu_n)$ admits a formulation in terms of transport \textit{maps}
\[ d_\infty(\nu, \nu_n)= \min_{T_{n \sharp} \nu = \nu_n }  \lVert Id - T_n\rVert_{L^\infty(\nu)},\]
where the $\min$ in the above formula is taken over all transportation maps between $\nu$ and $\nu_n$. The equivalence of these two formulations, as well as the existence and uniqueness of minimizers to these problems, are topics studied in \cite{Champion}. From now on we use $T_n$ to represent the optimal transport map between $\nu$ and $\nu_n$ and we let $\delta_n$ be
\begin{equation}
\delta_n := \lVert  Id - T_n \rVert_{L^\infty(\nu)}.
\label{deltan} 
\end{equation}
\begin{remark}
\label{Remarkinfty}
Let $Z$ be a finite subset of $\T^d$ with cardinality $k$ and let $\mu_k$ be the empirical measure associated to $Z$, i.e. all points in $Z$ are given the same amount of mass. Then, regardless of what $Z$ is, the $\infty$-OT distance between $\mu_k$ and $\nu$ satisfies the lower bound
\[ \frac{C_d}{k^{1/d}} \leq d_\infty(\nu, \mu_k),  \]
where $C_d$ is a constant only depending on dimension.  Indeed, if this was not the case, it would be possible to construct $k$ disjoint balls in $\T^d$ whose union has bigger $\nu$-measure than $\T^d$ itself.
\end{remark}

Let us now specify how the weights $\omega$ on $X_n \times X_n$ are defined in this setting, or in other words, how the point cloud $X_n$ is endowed with a weighted graph structure. First, let $\eta : [0,\infty) \rightarrow [0, \infty)$ be the function defined by
\begin{equation}
\eta(r):=\begin{cases}1 & \text{ if } r \leq 1 \\ 0 & \text{otherwise}, \end{cases} 
\label{eta}
\end{equation}
and let $\sigma_\eta$ be the quantity
\begin{equation}
 \sigma_\eta:= \int_{0}^\infty  \eta(r) r^{d+1} dr = \frac{1}{d+2} .
 \label{sigmaeta}
\end{equation}
For given $\veps>0$ we let $\eta_\veps: [0,\infty) \rightarrow [0, \infty)$ be the rescaled version of $\eta$
\[ \eta_\veps(r) := \frac{1}{\veps^d}\eta\left( \frac{r}{\veps} \right),\]
and define the weight function $w_{\veps_n}$ by
\[ w_{\veps_n}( x, y) := \eta_{\veps_n}(\lvert  x-y \rvert), \quad \forall x , y \in \T^d, \]
where we think of $\veps=\veps_n$ as a length scale that depends on the number of data points $n$.

\begin{remark}
Although for simplicity we focus on the kernel $\eta$ defined above, we anticipate that all our results are still true for more general kernels. In particular, our results continue to hold if we simply assume that $\eta$ is a non-negative, compactly supported, and non-increasing function. In  general the quantity  $\sigma_\eta$ must be interpreted as in the first equality in \eqref{sigmaeta}.  
\label{remarketa}
\end{remark}

In the remainder we use $\divergence_n$, $\nabla_n$, $\Delta_n$ to represent the divergence, gradient and Laplacian operators in the geometric graph setting described above. We use $u_n$ for arbitrary elements in $L^2(\nu_n)$ and $V_n$ for arbitrary elements in $\mathfrak{X}(X_n)$. The expressions for $\divergence_n$, $\nabla_n$, $\Delta_n$ can be explicitly written as
\[ \nabla_n u_n (\x_i, \x_j) = \frac{u_n(\x_j) - u_n(\x_i)}{\veps_n} , \quad u_n \in L^2(\nu_n),  \]
\[ \divergence_n V_n(\x_i) = \frac{1}{n\veps_n^{d+1}} \sum_{j=1}^n ( V_n(\x_i, \x_j)  - V_n(\x_j, \x_i) ) \eta\left(\frac{ \lvert \x_i- \x_j \rvert}{\veps_n} \right) , \quad V_n \in \mathfrak{X}(X_n), \]
\[ \Delta_n u_n (\x_i) =  , \quad u_n \in L^2(\nu_n).  \]

\nc

\subsection{Discrete Wasserstein distance}
\label{Sec:DisWass}

Let $(X, \omega)$ be an arbitrary weighted graph, $\gamma$ a probability measure on $X$ giving positive mass to every point in $X$, and $\veps$ a positive number. Inspired by the Benamou-Brenier formulation of optimal transport at the continuum level, one can define a discrete Wasserstein distance on the set $\mathcal{P}(X)$ by introducing discrete continuity equations and their actions. This has been done in \cite{Maas} in the context of finite Markov chains.

\begin{definition}
We say that $ t \in [0,1] \mapsto  (\rho_{t}, V_{t})$ satisfies the discrete continuity equation (in flux form) if 
\begin{enumerate}
\item  $\rho_{t} \in \mathcal{P}(X)$ and $V_t \in \mathfrak{X}(X)$ for every $t \in [0,1]$.
\item For every $\x \in X$,
\begin{equation}
 \frac{d}{dt}\rho_{t}(\x) + \divergence_N(V_{t})(\x) = 0. 
\end{equation}

\end{enumerate}
\end{definition}

\begin{remark}
Notice that the second condition in the above definition implies the first one provided that $\rho_{0} \in \mathcal{P}(X)$ and $\rho_{t}$ stays non-negative for every $t \in [0,1]$. Indeed, this follows from the fact that 
\[ \frac{d}{dt} \sum_{\x \in X} \rho_{t}(\x) \gamma(\x)= -\sum_{\x \in X} \divergence_N(V_{t})(\x) \gamma(\x)=  - \langle \divergence_N(V_{t}) , \mathds{1} \rangle_{L^2(\gamma)} = \langle V_{t} , \nabla_N \mathds{1} \rangle_{\mathfrak{X}(X)} =  0, \]
where $\mathds{1}$ represents the function that is identically equal to one. 
\label{RemarkNonNegative}
\end{remark}
In order to define the action associated to a solution to the discrete continuity equation $t \in [0,1] \mapsto(\rho_{t}, V_{t})$, we first need to specify a way to interpolate the values of a density function at different points in $X$ so as to induce a ``density" on the edges of the graph. We use an \textit{interpolating function} $\theta: [0, \infty) \times [0,\infty) \rightarrow [0,\infty)$ to do this. We make the following assumptions on $\theta$.

\begin{assumption}
The following holds:	
\begin{enumerate}
	\item Symmetry: $\theta(s,t) = \theta(t,s) $ for all $s,t$.
	\item Monotonicity: $\theta(r,t) \leq \theta(s,t)$ for all $r \leq s$ and all $t$.
	\item Positive homogeneity: $\theta(\lambda s , \lambda t) = \lambda \theta(s, t)$ for all $\lambda \geq 0$ and all $s,t$.
	\item Normalization: $\theta(1,1)=1$. Combined with the positive homogeneity this implies that $\theta(\lambda, \lambda) = \lambda$ for all $\lambda \geq 0$.
	\item Joint concavity: The function $(x,y) \mapsto \theta(x,y)$ is concave. 
	\item We assume that 
	\[ C_\theta:= \int_{0}^1 \frac{1}{\sqrt{\theta(1-t, t)}} dt < \infty.\]
\end{enumerate}
\label{assump}
\end{assumption}

\begin{remark}
Assumptions (1)-(5) are all quite standard conditions to impose on an interpolating function.  As explained in Chapter 3 in \cite{Maas}, condition (6) is an optimal requirement in the sense that it is a necessary and sufficient condition for the discrete Wasserstein distance on a two point space to be finite.
\end{remark}
\nc

\begin{remark} Typical examples of functions $\theta$ satisfying the above conditions are:
\begin{enumerate}
\item The arithmetic mean
\[ \theta_1(r,s) = \frac{r+s}{2}. \]
\item The logarithmic mean defined by 
\[ \theta_2(r,s) = \int_{0}^1 r^u s^{1-u}du = \frac{r-s}{\log(r) - \log(s)}.\] 
\end{enumerate} 
\label{InterpolationgFuncs}
\end{remark}

\begin{remark}
In all our main results the only assumptions we make on $\theta$ are (1)-(6). Nevertheless, we will give special attention to interpolating functions of the form
\[ \theta(r,s)= \frac{r-s}{f'(r) - f'(s)}, \] 
where $f$ is a convex function. This is due to the close connection of such interpolating functions with heat flows on graphs (see \cite{Maas}). 
\end{remark}

\begin{definition}

The \textit{total action} (with respect to the interpolating function $\theta$) associated to a solution $ t \in [0,1] \mapsto (\rho_{t} , V_{t}) $ to the discrete continuity equation is  
\begin{equation}
 \int_{0}^1 \mathcal{A}_N(\rho_{t}, V_{t})  dt,
\label{DiscreteAction}
\end{equation}
where 
\begin{equation*}
\label{DiscreteAction0}
  \mathcal{A}_N ( \rho_{t}, V_{t})  :=    \sum_{\x, \y \in X} \frac{( V_{t}(\x, \y)  )^2}{\theta(\rho_{t}(\x) , \rho_{t}(\y) )} w(\x, \y) \gamma(\x) \gamma(\y).
\end{equation*}
In the definition of $\mathcal{A}_N(\rho_{t}, V_{t})$ we use the convention that when $V_{t}$ and $\theta$ are both equal to zero, the quotient $V_t^2/\theta$ must be interpreted as zero. Also, if $V_t$ is different from zero and $\theta$ is zero, the whole expression should be interpreted as taking the value $+\infty$ (i.e. these are forbidden flows of mass). 
\end{definition}

\begin{definition}(Discrete Wasserstein distance)
Let $\rho, \tilde{\rho}\in \mathcal{P}(X)$ be two discrete densities. The discrete Wasserstein distance $W_N(\rho, \tilde{\rho})$ is defined by
\[ (W_N(\rho, \tilde{\rho}))^2 := \inf \int_{0}^{1}  \mathcal{A}_N(\rho_{t}, V_{t}) dt, \]
where the $\inf$ is taken over all solutions $t \in [0,1] \mapsto (\rho_{t}, V_{t})$ of the discrete continuity equation that satisfy
\[  \rho_{0}= \rho , \quad \rho_{1}=\tilde{\rho}.\]
\label{def:DiscreteWass}
\end{definition}

\begin{remark}
We remark that the $\inf$ in the definition of the discrete Wasserstein distance can be replaced by $\min$ (see \cite{Erbar2012}). In other words, we can always find a geodesic connecting two arbitrary discrete probability measures.
\end{remark}
\nc

We now consider two examples that illustrate the notions introduced before. The first example is related to the work in \cite{gigmaas} and can be interpreted as a discrete model for local geometry. The second example is used in Section 2 and can be seen as a discrete model for fully non-local geometry.

\begin{example}[Linear graph]
Let $X = \{ 0,  1/N , 2 /N , \dots, (N-1)/N \}$, and let $\gamma$ be the uniform distribution on $X$. Let $\omega : X \times X \rightarrow \R$ be the weight matrix with zeros in all entries except at consecutive points where instead we have
\[  \omega \left(\frac{i}{N}, \frac{i+1}{N} \right)  =  \omega \left(\frac{i}{N}, \frac{i-1}{N} \right) = \frac{1}{2}, \quad \forall i=0, \dots, N-1. \]
In the above we are identifying the numbers $0$ and $1$ so that in particular the set $X$ can be thought as a regular grid on the torus $\T^1$. We set $\veps=1$.

%
%
%

The work of \cite{gigmaas} shows that the discrete Wasserstein space on this linear graph is an approximation (after appropriate rescaling) in the Gromov-Hausdorff sense of the usual Wasserstein space over $\T^1$. The situation in this example contrasts with that in Example \ref{Example1} below.
\end{example}

\begin{example}[Complete graph]
\label{Example1}
Let $X$ be any finite set with $N$ elements, $\gamma$ be the uniform distribution on $X$, $\omega$ be the $N \times N$ transition matrix with $1/ N$ in all its entries, and $\veps=1$. After fixing an interpolating function $\theta$ satisfying the conditions (1)-(6), let $\mathcal{P}(\mathcal{K}_N)$ be the associated discrete Wasserstein space. 

The diameter of the space $\mathcal{P}(\mathcal{K}_N)$ is an important quantity that will be used in Section 2, and in what follows we find estimates for it. We use the notation $\diam_\theta(\mathcal{K}_N)$ to emphasize the role of $\theta$ in the estimates. Let us first suppose that $\theta$ has the form
\begin{equation}
\theta(r,s) = \frac{r-s}{f'(r) - f'(s)},
\label{SpecialFormTheta}
\end{equation}
for some function $f$ with $f''>0$ in $(0,\infty)$. Without the loss of generality we can assume $f$ satisfies $f(0)=f(1)=0$.  One simple way to obtain an upper bound for $\diam_\theta(\mathcal{P}(\mathcal{K}_n))$ is to use Talagrand's inequality in $\mathcal{P}(\mathcal{K}_N)$. The idea is to find bounds for the distance in terms of an entropy functional $\mathcal{H}_{N,f}$ that can be written in terms of $f$; the condition $f(1)=0$ guarantees that the entropy functional is non-negative. This approach relies on the existence of a positive lower bound for the Ricci curvature in the space $\mathcal{K}_N$; in turn, lower bounds on the Ricci curvature are defined in terms of convexity properties of the entropy functional $\mathcal{H}_{N,f}$. The previous discussion is presented in detail in \cite{Erbar2012}. Furthermore, explicit calculations in Section 5 in \cite{Erbar2012} show that
\[ \Ric(\mathcal{K}_N)\geq \frac{1}{2} >0, \]
a relation that when inserted in Talagrand's inequality (see Theorem 1.5 in \cite{Erbar2012}) implies that
\[  W_N(\rho, \tilde{\rho}) \leq W_N(\rho, \mathds{1}) + W_N( \mathds{1},\tilde{\rho}) \leq C \max \{  \sqrt{\mathcal{H}_{N,f}(\rho)},  \sqrt{\mathcal{H}_{N,f}(\tilde{\rho})}   \}, \forall \rho, \tilde{\rho} \in \mathcal{P}(\mathcal{K}_N) \]
where $\mathcal{H}_{N,f}$ is defined by
\[  \mathcal{H}_{N,f}(\rho) := \frac{1}{N}\sum_{\x \in X} f( \rho(\x) ), \quad \rho \in \mathcal{P}(\mathcal{K}_N),\]
and where $C$ is a universal constant. Now, the convexity of the function $f$ implies that the maximum of $\mathcal{H}_{N,f}$ is achieved at extremal points. This fact combined with the fact that $f(0)=0$ implies that
\[  \mathcal{H}_{N,f}(\rho) \leq \frac{f(N)}{N},  \quad  \forall \rho \in \mathcal{P}(\mathcal{K}_N), \]
from where it follows that
\begin{equation}
\diam_\theta(\mathcal{P}(\mathcal{K}_N)) \leq   C \sqrt{\frac{f(N)}{N}}, 
\end{equation}
where again $C$ is a universal constant. Notice that the above general strategy allows us to find upper bounds for $\text{diam}_{\theta_2}(\mathcal{K}_N)$ where $\theta_2$ is the logarithmic mean(there $f(r) = r\log(r)$). More explicitly, we have
\begin{equation}
\diam_{\theta_2}(\mathcal{P}(\mathcal{K}_N)) \leq C \sqrt{ \log(N)}.
\label{EstimateLogarithmicMean}
\end{equation}

Notice that if the interpolating function $\theta$ does not have the form \eqref{SpecialFormTheta}, but still we can find an interpolating function $\theta_f $ of the form \eqref{SpecialFormTheta} that is smaller than $\theta$, then
\[ \diam_{\theta}(\mathcal{P}(\mathcal{K}_N)) \leq  \diam_{\theta_f}(\mathcal{P}(\mathcal{K}_N)) \leq C \sqrt{ \frac{f(N)}{N}} . \]
This can be seen directly from the definitions of discrete Wasserstein space given that the action $\mathcal{A}_N$ is monotone with respect to the interpolating function. 

%

\end{example}
\nc

\subsubsection{Continuity equation in the geometric graph setting}

We finish this section by writing explicitly the continuity equation and its associated action in the geometric graph setting. The equation reads
\[ \frac{d}{dt} \rho_{n,t}(\x_i) + \frac{1}{n\veps_n^{d+1}}  \sum_{j=1}^n (V_{n, t}(\x_i, \x_j) - V_{n,t}(\x_j , \x_i) ) \eta\left( \frac{\lvert \x_i - \x_j \rvert}{\veps_n} \right) =0 \]
and the corresponding action is 
\[\mathcal{A}_n(\rho_{n,t}, V_{n,t}) = \frac{1}{n^2 \veps_n^d} \sum_{i,j} \frac{(V_{n,t}(\x_i, \x_j))^2}{\theta(\rho_{n,t}(\x_i), \rho_{n,t}(\x_j))}\eta \left( \frac{\lvert \x_i - \x_j \rvert}{\veps_n} \right).\]

\begin{remark}
We notice that the set $\mathcal{P}(X_n)$ is a subset of $\mathcal{P}(\T^d)$ so that in particular we may compute $W(\mu_n, \tilde{\mu}_n) $ for two measures $\mu_n, \tilde{\mu}_n \in \mathcal{P}(X_n)$. Nevertheless, in the remainder we will mostly think of $\mathcal{P}(X_n)$ as endowed with the discrete Wasserstein distance introduced in this section. 
\end{remark}

\begin{remark} An important observation is that geometric graphs behave, at lenght scale smaller than $\veps$, like complete graphs. This is because in an Euclidean ball $B$ of radius $\veps/2$, all points in $X_n \cap B$ are connected to each other. The estimates we found in Example \ref{Example1} will come in handy later on.
\label{GeoComplete}
\end{remark}

\nc

\subsection{Gromov-Hausdorff distance}
\label{Sec:Gromov}

In this section we consider $(\mathcal{X}, d_{\mathcal{X}})$ and $(\mathcal{Y}, d_{\mathcal{Y}})$ two arbitrary metric spaces and we define the Gromov-Hausdorff distance between them. We do this via the notion of correspondences and refer the interested reader to Chapter 7.3 in \cite{Burago} for a more complete discussion on the matter.

\begin{definition}
A correspondence $\mathfrak{R}$ between $\mathcal{X}$ and $\mathcal{Y}$ is a subset of the cartesian product $\mathcal{X} \times \mathcal{Y}$ such that for all $x \in \mathcal{X}$ there exists $y \in \mathcal{Y}$ for which $(x,y) \in \mathcal{R}$ and also for every $y \in \mathcal{Y}$ there exists $x \in \mathcal{X}$ for which $(x,y) \in \mathcal{R}$.
\end{definition}

\begin{definition}
Given a correspondence $\mathfrak{R}$ between $\mathcal{X}$ and $\mathcal{Y}$, we define its distortion $dis \mathfrak{R}$ by
\[ dis \mathfrak{R} := \sup  \lvert d_{\mathcal{X}}(x,x')  - d_{\mathcal{Y}}(y,y') \rvert,  \]
where the $\sup$ is taken over all pairs $(x,y), (x', y') \in \mathfrak{R}$.
\end{definition}

\begin{definition}
The Gromov-Hausdorff distance between metric spaces $\mathcal{X}$ and $\mathcal{Y}$ is defined by
\[  d_{GH}(\mathcal{X}, \mathcal{Y}) := \frac{1}{2} \inf_{\mathfrak{R}}  dis \mathfrak{R},\] 
where the infimum is taken over all correspondences $\mathfrak{R}$ between $\mathcal{X}$ and $\mathcal{Y}$. 

Also, we say that a sequence of metric spaces $\{ (\mathcal{X}_n , d_{\mathcal{X}_n}) \}_{n \in \N}$ converges in the Gromov-Hausdorff sense towards $(\mathcal{X}, d_{\mathcal{X}})$ as $n \rightarrow \infty$ (and write  $\mathcal{X}_n \converges{GH} \mathcal{X}$), if 
\[\lim_{ n\rightarrow \infty} d_{GH}(\mathcal{X}_n, \mathcal{X})=0.   \]
\label{GromovHausdorff}
\end{definition}

\begin{remark}
The term ``distance" in the above discussion has to be understood informally. First, there is a set theoretic obstruction as we can not talk about the ``set of all metric spaces". Secondly, even if we could talk about such set, $d_{GH}$ would only be a distance on the set of isometry classes.  Both of the above obstructions can be avoided if one considers the space of isometry classes of compact metric spaces.
\end{remark}

One way to estimate the Gromov-Hausdorff distance between two metric spaces is to find a $\varrho$-\textit{isometry} between them.
\begin{definition}
Given two metric spaces $\mathcal{X}$ and $\mathcal{Y}$, a map $F : \mathcal{X} \rightarrow \mathcal{Y} $ is said to be a $\varrho$-isometry if
\begin{enumerate}
	\item  $\sup_{x, x' \in \mathcal{X}}   \lvert  d_{\mathcal{Y}}(F(x), F(x')) - d_{\mathcal{X}}(x, x') \rvert \leq \varrho  $
	\item For every $y \in \mathcal{Y}$ there exists $x \in \mathcal{X}$ such that $ d_{\mathcal{Y}} (y , F(x))  \leq \varrho$.
\end{enumerate}

The first condition says that the \textit{distortion} of $F$ is at most $\varrho$, and the second condition says that the set $F(\mathcal{X})$ is a \textit{$\varrho$-net} in $(\mathcal{Y}, d_{\mathcal{Y}})$.

\end{definition}

It is straightforward to check that if one can find a $\varrho$-isometry between $\mathcal{X}$ and $\mathcal{Y}$, then
\[ d_{GH}(\mathcal{X}, \mathcal{Y}) \leq 2 \varrho.\] 
From this observation it follows that it suffices to construct a $\varrho$-isometry between two metric spaces to bound their Gromov-Hausdorff distance.

\subsection{Main results}
\label{Sec:main}

The main results of the paper are the following.

\begin{theorem}
	\label{thm:main}
For each $n \in \N$ let $X_n$ be a subset of $\T^d$ with cardinality $n$ and denote by $\nu_n$ the empirical measure associated to $X_n$. Let $T_n$ be the optimal, in the $\infty$-OT sense, transportation map between $\nu$ (the uniform distribution on $\T^d$) and $\nu_n$, and let $\delta_n$ be 
\[ \delta_n := \lVert T_n - Id \rVert_\infty.  \]
We assume that $\delta_n \rightarrow 0$ as $n \rightarrow \infty$. Let $\eta$ be as in \eqref{eta} and construct the distance $W_n$ on $\mathcal{P}(X_n)$ as in Definition \ref{def:DiscreteWass} where $\theta$ satisfies conditions (1)-(6) in Subsection \ref{Sec:DisWass} and $\veps_n$ satisfies
\begin{equation}
 \veps_n \rightarrow 0 , \quad \frac{\delta_n}{\veps_n}\rightarrow 0, \quad \veps_n\diamP \rightarrow 0 ,  \quad \text{ as } n \rightarrow \infty,
 \label{Scalingsveps}
 \end{equation}
where $\diamP$ is the diameter of the discrete Wasserstein space on the complete graph with $N=n\epsilon_n^d$ points as defined in Example \ref{Example1}. Then,
\[ (\mathcal{P}(X_n), W_n) \converges{GH} (\mathcal{P}(\T^d),
 \frac{1}{\sqrt{\alpha_d\sigma_\eta}}W), \quad \text{as } n \rightarrow \infty ,\]
where $\alpha_d$ is the volume of the unit ball in $\R^d$, $\sigma_\eta$ is as in \eqref{sigmaeta} and $W$ is the Wasserstein metric on $\mathcal{P}(\T^d)$.

In particular, when the interpolating function $\theta$ is the logarithmic mean, the conditions on $\veps_n$ read
\begin{equation*}
\veps_n \rightarrow 0 , \quad \frac{\delta_n}{\veps_n}\rightarrow 0, \quad \veps_n\sqrt{\log(n \veps_n^d)} \rightarrow 0 ,  \quad \text{ as } n \rightarrow \infty.
\end{equation*}
\end{theorem}

We now specify the above theorem to the case in which $\x_1, \dots, \x_n$ are i.i.d. samples from the uniform distribution on $\T^d$.

\begin{corollary}
In the context of Theorem \ref{thm:main} suppose that $X_n = \{\x_1, \dots, \x_n \}$ where $\x_1, \dots, \x_n , \dots $ are i.i.d. random variables uniformly distributed on $\T^d$ where $d\geq 2$. Then, with probability one,
    \[ (\mathcal{P}(X_n), W_n) \converges{GH} (\mathcal{P}(\T^d),
    \frac{1}{\sqrt{\alpha_d\sigma_\eta}}W), \quad \text{as } n \rightarrow \infty ,\]
 provided that 
 \begin{equation*}
 \veps_n \rightarrow 0 , \quad \frac{\log(n)^{p_d}}{n^{1/d}\veps_n}\rightarrow 0, \quad \veps_n\diamP \rightarrow 0 ,  \quad \text{ as } n \rightarrow \infty,
 \end{equation*}   
 where $p_2=3/4$ and $p_d = 1/d$ for $d\geq 3$.
\end{corollary}
The corollary follows directly from Theorem \ref{thm:main} when combined with the rates of convergence of empirical measures in the $\infty$-OT distance studied in \cite{W8L8} (see also the references within).

\begin{remark}
As was anticipated in Remark \ref{remarketa}, Theorem \ref{thm:main} and its corollary still hold if the weights $w_{\veps_n}$ are defined using any non-increasing, compactly supported kernel $\eta$. We also remark that $\frac{\delta_n}{\veps_n} \rightarrow 0$ implies that $n \veps_n^d \rightarrow \infty$ as $n \rightarrow \infty$ (see for example the first paragraphs in the proof of Proposition \ref{SmoothDiscrete} below), and so condition $\veps_n\diamP \rightarrow 0$ forces $\veps$ to go to zero with $n$ fast enough. Qualitatively speaking this condition guarantees that the complete-graph-like behavior of the geometric graph at length scale $\veps$ disappears in the limit $n \rightarrow \infty$. 
\end{remark}

\section{Preliminaries}

\subsection{A priori estimates for the discrete Wasserstein distance}
\label{Sec:Apriori}

In this section we establish some a priori estimates for the discrete Wasserstein distance. These estimates will be used in the remainder.

\begin{proposition} For every $\mu_n , \tilde{\mu}_ n \in \mathcal{P}(X_n)$ we have
	\[  W_n( \mu_n, \tilde{\mu}_n) \leq C W(\mu_n, \tilde{\mu}_n ) + C \veps_n  \diam_\theta(  \mathcal{P}(\mathcal{K}_N)  ), \]
	where $N= n \veps_n^d$, $\diam_\theta(  \mathcal{P}(\mathcal{K}_N)  )$ is the diameter of the discrete Wasserstein space from Example \ref{Example1} and  $C>0$ is a constant that depends only on $\theta$ and $d$.  		
	\label{AprioriBound}
\end{proposition}

\begin{remark}
The above proposition says that the discrete Wasserstein distance is controlled by the usual Wasserstein distance \textit{plus} an extra term which can be interpreted as a \textit{non-local effect} term.
\end{remark}

\nc
Before proving Proposition \ref{AprioriBound} let us introduce some notation and establish some preliminary results. For every $i=1, \dots, n$ let $\mu_{n}^i \in \mathcal{P}(X_n)$ be the uniform distribution on $B(\x_i , \veps_n / 8 ) \cap X_n $, where $B(\x_i, \veps_n/ 8)$ is the Euclidean ball centered at $\x_i$ with radius $\veps_n /8$. We let $X^i:= \supp(\mu_n^i)$ and let $N_i$ be the number of elements in $X^i$. Observe that thanks to the assumption $\frac{\delta_n}{\veps_n} \rightarrow 0$ we have
 \begin{equation}
 cn \veps_n^d \leq N_i \leq C n \veps_n^d, \quad \forall i=1, \dots, n
 \label{EstimateNi}
 \end{equation}
 where $c,C$ are constants independent of $n$. We split the proof of Proposition \ref{AprioriBound} into three lemmas.

\begin{lemma}
	\label{NonLocaleffect}
For every $i=1, \dots, n$ we have
\[  W_n(\delta_{\x_i}, \mu_n^i)  \leq C \veps_n \diam_\theta(\mathcal{P}(\mathcal{K}_{N_i})),  \]
where $C$ is a constant independent of $i$ or $n$, and where $\diam_\theta(\mathcal{P}(\mathcal{K}_{N_i}))$ is defined in Example \ref{Example1}.
\end{lemma}

\begin{proof}

Consider the point cloud $X^i $. Notice that 
\[ \omega_{\veps_n}(\y, \z) =\frac{1}{\veps_n^d}, \quad \forall \y , \z \in X^i\]
and in particular the weighted graph $(X^i , \omega_{\veps_n})$ is simply a rescaled version of the complete graph with $N_i$ nodes. This suggests that in order to estimate $W_n(\delta_i , \mu_n^i)$ we should compute the action of a rescaled version of a geodesic in $\mathcal{P}(\mathcal{K}_{N_i})$  connecting a point mass with the uniform distribution on $X^i$. Consider then a curve $t \in [0,1] \mapsto (\tilde{\rho}_t, \tilde{V}_t)$ with $\tilde{\rho}_t \in \mathcal{P}(\mathcal{K}_{N_i})$ for all $t \in [0,1]$, $\tilde{\rho}_0$ is the discrete density of a point mass at $\x_i$,  $\tilde{\rho}_1$ is the discrete density of the uniform distribution on $X^i$, 
\[ \frac{d}{dt} \tilde{\rho}_t(\y) + \frac{1}{N_i} \sum_{\z \in X^i}( \tilde{V}_t(\y, \z) - \tilde V_t(\z, \y) )=0 , \quad \forall \y \in X^i, \] 
and corresponding total action satisfying
\[ \frac{1}{N_i^2} \sum_{\y, \z \in X^i} \int_{0}^1  \frac{\tilde{V}_t(\y, \z)^2 }{\theta(\tilde{\rho}_t(\y), \tilde{\rho}_t (\z) )} dt \leq \diam_\theta(\mathcal{P}(\mathcal{K}_{N_i}))^2 . \]

We define $\rho_{n,t} \in \mathcal{P}(X_n)$ by
\[ \rho_{n,t}(\y):= \frac{  n \tilde{\rho}_{t}(\y)  }{N_i}, \quad \y \in X^i, \]
and we set $\rho_{n,t}$ to be zero outside of $X^i$. Similarly, we define the discrete vector field $V_{n,t}$ by
\[  V_{n,t}(\y, \z) :=  \frac{n^2 \veps_n^{d+1}}{N_i^2}  \tilde{V}_t(\y, \z) , \quad (\y, \z) \in X^i \times X^i, \]
and we set $V_{n,t}$ to be zero outside of $X^i \times X^i$. With these definitions it is then straightforward to check that $t \mapsto (\rho_{n,t}, V_{n,t})$ satisfies the discrete continuity equation, that $\rho_{n,0}$ is the discrete density of $\delta_{\x_i}$, that $\rho_{n,1}$ is the discrete density of $\mu_n^i$, and finally that
\[ W_n^2(\delta_{\x_i}, \mu_n^i) \leq \int_{0}^{1} \mathcal{A}_{n}(\rho_{n,t}, V_{n,t}) dt  \leq  \left( \frac{n \veps_n^d}{N_i} \right) \veps_n^2 (\diam_\theta( \mathcal{K}_{N_i}) )^2 \leq C \veps_n^2  (\diam_\theta(\mathcal{K}_{N_i}) )^2,  \]
where the last inequality follows from \eqref{EstimateNi}. The desired estimate now follows.

\end{proof}

\nc

Next, we use use the previous result to bound  $W_n(\delta_{\x_i},\delta_{\x_j})$ for points $\x_i$ and $\x_j$ that are within distance $\veps_n$ from each other.  To estimate $W_n(\delta_{\x_i},\delta_{\x_j})$ for points $\x_i$, $\x_j$ that are far away from each other we use a discretization of the Euclidean shortest path connecting the points. 
\nc

 \begin{lemma}
 	Let $\hat{\alpha}$ be the real number in $[0, \pi/2]$ for which $\sin(\hat{\alpha} ) = 1/8$. Let $x \in \T^d$ and let $\vec{u}$ be a unit vector in $\R^d$. 
 	Consider the set 
 	\[  \mathcal{C}_n(x, \vec{u}) := \{ y \in \T^d \: : \: \frac{\veps_n}{4}< \lvert y - x\rvert < \frac{\veps_n}{2} , \quad \alpha_{y} \in [0, \hat{\alpha}]    \}, \]
 	where $\alpha_y$ is the angle between the vectors $(y-x)$ and $\vec{u}$. Then, 
 	\[  \mathcal{C}_n(x, \vec{u}) \cap X_n \not = \emptyset .\]
 	\label{LemmaAuxGraphDistance1}
 \end{lemma}
 \begin{proof}
 	To show this, let $\tilde{x}:= x + \frac{3 \veps_n}{8}\vec{u}$. It is straightforward to see that 
 	\[  dist(\tilde{x}, \partial \mathcal{C}_n ) \geq \frac{3 \veps_n}{8} \sin(\hat{\alpha}) \geq \frac{\veps_n}{24} > \delta_n,\]
 	where the last inequality follows from the assumptions on $\veps_n$. If $X_n \cap  \mathcal{C}_n(x, \vec{u}) $ was the empty set, this would contradict the fact that $\delta_n$ is the $\infty$-transportation distance between $\nu$ and $\nu_n$.
 \end{proof}

\begin{lemma}Let $\x_i, \x_j \in X_n$ be such that $  \lvert \x_i - \x_j \rvert < \frac{\veps_n}{2}$. Then, 
\[  W_n(\mu_n^i, \mu_n^j) \leq C_\theta \veps_n,\]
for some finite constant $C_\theta>0$ that depends on the interpolating function $\theta$. 
\label{LemmaApriori3}
\end{lemma}
\begin{proof}
By Lemma \ref{LemmaAuxGraphDistance1} and the triangle inequality it is enough to prove the result assuming that $ \frac{\veps_n}{4} < \lvert \x_i - \x_j \rvert$. In this case we have $X^i \cap X^j = \emptyset$ (i.e. the supports of the measures $\mu_n^i$ and $\mu_n^j$ are disjoint). We construct a path connecting $\mu_n^i$ and $\mu_n^j$ such that at every point in time, the mass in $X^i$ is uniformly distributed, and the same is true for $X^j$.  The mass between the ``blubs'' $X^i$ and $X^j$ is exchanged following the construction for the $2$-point space from Section 2 in \cite{Maas}. Points in $X_n$ outside of $X^i\cup X^j$ play no role in the transportation scheme.  More precisely, take $p=q=1/2$ in Lemma 2.3 and Theorem 2.4 in \cite{Maas}. It follows that there exist differentiable functions $\gamma : [0,1] \rightarrow [-1,1]$ and $\chi: [0,1 ] \rightarrow \R$ such that
\[ \gamma_0 = -1, \quad \gamma_1=1, \]
\[ \gamma_t' =\frac{1}{4}\theta(1+\gamma_t, 1- \gamma_t)  \chi_t, \quad \forall t \in (0,1), \]
and
\[ \int_{0}^{1}\chi_t^2 \theta(1+\gamma_t, 1- \gamma_t) dt  < \infty. \]
We notice that (6) in Assumption \ref{assump} is crucial for this last inequality to be true. For every $t \in [0,1]$ let

\[  \rho_{n,t}(\x) := \begin{cases}  \frac{n(1-\gamma_t)}{2 N_i } , & \x \in X^i \\
\frac{n(1+\gamma_t)}{2 N_j}, & \x \in X^j, \\ 0 , & \x \in (X^i\cup X^j)^c.    \end{cases} \]
A straightforward computation shows that
\[ \frac{d }{dt} \rho_{n,t}(\x)  + \divergence_n(V_{n,t})(\x) =0 , \quad \forall \x \in X_n, \]
where
\[ V_{n,t}(\x,\y) :=  \begin{cases} \frac{n^2 \veps_n^{d+1}}{16 N_i N_j} \theta(1+\gamma_t, 1- \gamma_t)  \chi_t  & (\x,\y) \in X^i \times X^j \\ -V_{n,t}(\y,\x) & (\x,\y) \in X^j \times X^i   \\ 0 & \text{otherwise}. \end{cases}  \]
The curve $t \in [0,1] \mapsto \rightarrow (\rho_{n,t}, V_{n,t})$ is thus a solution to the discrete continuity equation connecting the measures $\mu_n^i$ and $\mu_n^j$, and so
\begin{align*}
\begin{split}
 W_n^2(\mu_n^i, \mu_n^j) \leq \int_{0}^{1} \mathcal{A}_n(\rho_t, V_t) dt & = \frac{2}{n^2\veps_n^d} \sum_{\x \in X^i}\sum_{\y \in X^j }  \frac{n^4 \veps_n^{2(d+1)}}{16^2 N_i^2 N_j^2} \int_{0}^1 \frac{\theta(1+\gamma_t, 1-\gamma_t) ^2}{\theta(\rho_{n,t}(\x) , \rho_{n,t}(\y) )} \chi_t^2 dt  
 \\& \leq C \frac{n \veps_n^d \veps_n^2 \max \{N_i, N_j \}}{N_i N_j}  \int_{0}^1 \theta(1+\gamma_t, 1-\gamma_t)  \chi_t^2 dt 
 \\& \leq C \veps_n^2 \int_{0}^1 \theta(1+\gamma_t, 1-\gamma_t)  \chi_t^2 dt,
\end{split}
\end{align*}
where in the first inequality we have used the homogeneity and monotonicity of $\theta$ to write
\[ \frac{1}{\theta(\rho_{n,t}(\x) , \rho_{n,t}(\y) )} \leq \frac{2\max\{ N_i, N_j \}}{n} \frac{1}{\theta(1+\gamma_t, 1-\gamma_t)}, \] 
and in the last inequality we have used \eqref{EstimateNi}. The desired inequality follows.
\end{proof}
\nc

We are now in a position to establish Proposition \ref{AprioriBound}.

\begin{proof}[Proof of Proposition \eqref{AprioriBound}]
\textbf{Step 1:} We first establish the result in the case where $\mu_n, \tilde{\mu}_n$ are arbitrary point masses $\mu_n := \delta_{\x_i},\quad \tilde{\mu}_n := \delta_{\x_j}$.  We claim that there exists a sequence $x_0, x_1, \dots, x_k$ in $X_n$ satisfying the following properties:
\begin{enumerate}
	\item $x_0=\x_i$, $x_k=\x_j$.
	\item $\lvert x_{l+1} - \x_j \rvert < \lvert  x_{l} - \x_j \rvert$ for every $l$.
	\item $\frac{1}{2}  \lvert  x_l-x_{l+1}  \rvert  \leq \lvert x_l - \x_j \rvert  - \lvert x_{l+1} - \x_j  \rvert$ for every $l$.
	\item $x_{l+1} \in \mathcal{C}_n(x_l, \frac{\x_j - x_l}{\lvert \x_j - x_l \rvert})$ for every $l=0, \dots, k-2$.
	\item $\lvert x_{k-1} - \x_j  \rvert < \frac{\veps_n}{2} $,
\end{enumerate}
where $\mathcal{C}_n$ is as in Lemma \ref{LemmaAuxGraphDistance1}.

To see this, first set $x_0:= \x_i$.  If $\lvert x_0 - \x_j \rvert < \frac{\veps_n}{2} $ then we simply set $x_1= \x_j$ and stop the construction. Otherwise, we know by Lemma \ref{LemmaAuxGraphDistance1} that there exists $x_1 \in  \mathcal{C}_n(x_0, \frac{\x_j-x_0}{\lvert  \x_j -x_0 \rvert}) \cap X_n$. We claim that 
\begin{equation}
\frac{1}{2} \lvert  x_0-x_1  \rvert  \leq \lvert x_0 - \x_j \rvert  - \lvert x_1 - \x_j  \rvert. 
\label{auxpath}
\end{equation}
 Indeed, notice that if $\lvert x_0- x_1 \rvert > 4 \lvert x_1 -\x_j \rvert$ then,
\[  \frac{1}{2} \lvert x_0- x_1 \rvert  < \lvert x_0 - x_1 \rvert -2 \lvert x_1 - \x_j \rvert \leq \lvert x_1 - \x_j \rvert + \lvert \x_j - x_0  \rvert - 2\lvert x_1 - \x_j \rvert  = \lvert x_0 - \x_j \rvert- \lvert x_1 - \x_j \rvert,  \]
and so \eqref{auxpath} holds in this case.  On the other hand, if $\lvert x_0- x_1 \rvert \leq 4 \lvert x_1 -\x_j \rvert$ we may consider the triangle in the figure below and notice that

\begin{figure}[h]
\includegraphics[width=0.5\textwidth, trim={0 2cm 0 1cm} ]{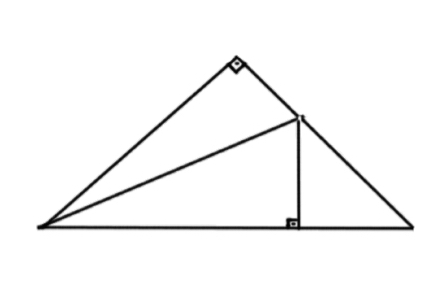}
\put(-70,45){$\theta_2$}
\put(-88,50){$\theta_1$}
\put(-93,65){$\beta$}
\put(-182,13){$\alpha$}
\put(-220,7){$x_0$}
\put(-72,68){$x_1$}
\put(-15,7){$\x_j$}
\label{Figure1}
\end{figure}

\[  \cos(\theta_2) = \frac{\sin(\alpha) \lvert x_0-x_1 \rvert }{\lvert x_1 - \x_j \rvert}  \leq 4 \sin(\alpha) \leq 1/2,  \] 
where the last inequality follows from the fact that $\sin(\alpha) \leq 1/8$ given that $x_1 \in \mathcal{C}_n(x_0, \frac{ \x_j-x_0}{\lvert  \x_j - x_0 \rvert}) $; in particular  $\theta_2 \geq  \pi /3$. It follows that $0 \leq \beta = \pi - \theta_1 -\theta_2 = \frac{\pi}{2} + \alpha - \theta_2 \leq \frac{\pi}{2} + \frac{\pi}{6} - \frac{\pi}{3}  =   \frac{\pi}{3}  $  and thus 
\[ \cos(\beta)  \geq \cos( \pi /3) =1/2 .\]
Therefore,
\[  \frac{1}{2}\lvert x_0 - x_1 \rvert  + \lvert x_1 - \x_j  \rvert \leq \cos(\beta) \lvert x_0 - x_1 \rvert  + \lvert x_1 - \x_j  \rvert     \leq  \lvert x_0 - \x_j \rvert. \]
This establishes \eqref{auxpath}.

We continue constructing the sequence $x_0, x_1 , x_2, \dots$ obtaining $x_{l+1}$ from $x_l$ in the same way we obtained $x_1$ from $x_0$. Given that $\lvert x_{l} - \x_j \rvert  $ decreases at every new step and given that $X_n$ is finite, we conclude that after a finite number of steps our construction will arrive to $\x_j$.

Now for every $l=0, \dots, k$ let $\mu_n^l$ be the uniform distribution on $B(x_l , \frac{\veps_n}{8}) \cap X_n$. We can then use the properties of the path $x_0, x_1, \dots, x_k$ and Lemma \ref{LemmaApriori3} to obtain
\begin{align*}
\begin{split}
W_n(\mu_n^{0}, \mu_n^{k}) & \leq \sum_{l=0}^{k-1} W_n(\mu_n^{l}, \mu_n^{l+1})
\\& \leq  C_\theta \sum_{l=0}^{k-1} \veps_n 
\\& \leq 4C_\theta \sum_{l=0}^{k-2} \lvert x_l  - x_{l+1} \rvert +  C_\theta\veps_n     
\\&  \leq \frac{8C_\theta}{2} \sum_{l=0}^{k-1} \lvert x_l  - x_{l+1} \rvert + C_\theta\veps_n
\\& \leq 8C_{\theta} \left( \frac{1}{2} \sum_{l=0}^{k-2} \lvert x_l  - x_{l+1} \rvert + \frac{1}{2} \lvert x_{k-1} - x_k \rvert\right) + C_\theta\veps_n
\\& \leq 8C_{\theta} \left(\sum_{l=0}^{k-2} ( \lvert x_l - \x_j  \rvert - \lvert x_{l+1} - \x_j \rvert  )   +    \lvert x_{k-1} - x_k \rvert\right) + C_\theta\veps_n  ,
\\&= 8C_\theta\lvert x_0- \x_j \rvert + C_\theta\veps_n
\end{split}
\end{align*}
where the last equality follows from the fact that the sum in the second to last line is telescoping. Finally, the triangle inequality, the above estimate,  Lemma \ref{NonLocaleffect}, and \eqref{EstimateNi} imply that
\begin{align*}
 W_n(\delta_{\x_i}, \delta_{\x_j}) & \leq W_n(\delta_{x_0}, \mu_n^0) + W_n(\mu_n^0, \mu_n^k) + W_n(\mu_n^k , \delta_{x_k}) 
 \\ &  \leq C_\theta \lvert x_0- \x_j \rvert  + C_\theta \veps_n \diam_\theta(\mathcal{P}( \mathcal{K}_N ) )
 \\& =  C_\theta W(\delta_{\x_i}, \delta_{\x_j})  + C_\theta \veps_n \diam_\theta(\mathcal{P}( \mathcal{K}_N ) ),
\end{align*}
where we recall $N= n \veps_n^d$.

\textbf{Step 2:} Following Proposition 2.14 in \cite{Erbar2012} we know that for general $\mu_n, \tilde{\mu}_n \in \mathcal{P}(X_n)$, we have the one sided relation
\[ W^2_n( \mu_n , \tilde{\mu}_n) \leq \min_{\pi \in \Gamma( \mu_n, \tilde{\mu}_n ) } \int \int W_n^2(\delta_{\x_i}, \delta_{\x_j} ) d\pi(\x_i, \x_j).  \]
In contrast, for the usual Wasserstein distance we have
\[ W^2( \mu_n , \tilde{\mu}_n) = \min_{\pi \in \Gamma( \mu_n, \tilde{\mu}_n ) } \int \int  \lvert \x_i - \x_j \rvert^2 d \pi(\x_i, \x_j).   \]
We can combine the above relations together with Step 1 to conclude that
\[  W_n(\mu_n, \tilde{\mu}_n)   \leq C  W(\mu_n, \tilde{\mu}_n) + C\veps_n \diamP, \]
for some constant $C>0$ only depending on $d$, and $\theta$. This establishes the desired result. 
\end{proof}

We finish this section with some bounds for the distance $W$ (or $W_n$) between measures, in terms of the uniform difference of the corresponding densities.


\begin{lemma} 
Let $\rho_0$ and $\rho_1$ be two densities with respect to $\nu$, satisfying that
\[ 0< a \leq \rho, \tilde\rho \leq M.\]
Then, 
\[ W(\rho dx , \tilde\rho dx) \leq C \left(\frac{1}{a} + M \right) \lVert\rho - \tilde\rho \rVert_{L^\infty(\T^d)}. \]
\label{W2L8}
\end{lemma}

\begin{proof}
The result follows immediately from the proof of Theorem 1.2 in \cite{W8L8} by noticing that any two measures in $\mathcal{P}(\T^d)$ can be seen as measures in $\mathcal{P}([0,1]^d)$ and that their Wasserstein distance in $\mathcal{P}(\T^d)$ is smaller than their distance in $\mathcal{P}([0,1]^d)$.  
\end{proof}

\begin{lemma} Let $\rho_n$ and $\tilde{\rho}_n $ be two discrete densities and assume that
	\[0< a \leq \rho_n(\x_i), \tilde{\rho}_n(\x_i) \leq M , \quad \forall i=1, \dots, n. \]
 Then,
	\[ W_n(\rho_n, \tilde{\rho}_n) \leq C\left(\frac{1}{a} + M \right) \lVert\rho_n - \tilde\rho_n \rVert_{L^\infty(\nu_n)} + C  \veps_n \diamP  . \]
	\label{DiscreteW2L8}
\end{lemma}
\nc
\begin{proof}
Let $T_n$ be the optimal, in the $\infty$-OT sense, transportation map between $\nu$ and $\nu_n$. For every $i=1, \dots, n$ we let $U_i := T_n^{-1}(\x_i)$. Notice that all the sets $U_i$ have $\nu$ measure equal to $1/n$ and that the distance between $\x_i$ and any point in $U_i$ is smaller than $\delta_n$ (the $\infty$-OT distance between $\nu$ and $\nu_n$).

Let $\rho_n, \tilde{\rho}_n$ be two discrete densities bounded away from zero. We consider the measures $\mu, \tilde \mu$ in $\mathcal{P}(\T^d)$ with densities
 \[ \rho(x) := \rho_n(\x_i), \quad \forall x \in U_i, \quad\forall i=1, \dots, n, \]
 \[ \tilde{\rho}(x) := \tilde{\rho}_n(\x_i), \quad \forall x \in U_i,\quad \forall i=1, \dots, n.\]
 We denote by $\mu_n$ and $\tilde{\mu}_n$ the measures with discrete densities $\rho_n$ and $\tilde{\rho}_n$. Note that
 \[  W(\mu_n, \mu) \leq \delta_n  \quad \text{and } \quad  W( \tilde{\mu}_n , \tilde{\mu}) \leq \delta_n,\]
since one can transport one of the measures into the other by redistributing the mass within each of the cells $U_i$. On the other hand, from Lemma \ref{W2L8} we know that  
	\[   W(\mu, \tilde{\mu}) \leq   C\left(\frac{1}{a} + M \right) \lVert\rho - \tilde\rho \rVert_{L^\infty(\T^d)}= C\left(\frac{1}{a} + M \right) \lVert \rho_n- \tilde{\rho}_n  \rVert_{L^\infty(X_n)}.\]
Hence, from Proposition \eqref{AprioriBound} we conclude that
\begin{align*}
\begin{split}
W_n(\mu_n, \tilde{\mu}_n) & \leq  C W(\mu_n, \tilde{\mu}_n) +C\veps_n \diamP  
\\& \leq C W(\mu_n, \mu) + CW(\mu, \tilde{\mu}) + C W(\tilde{\mu}, \tilde{\mu}_n) + C \veps_n \diamP 
\\& \leq C\left(\frac{1}{a} + M \right) \lVert \rho_n- \tilde{\rho}_n  \rVert_{L^\infty(X_n)} + C\veps_n \diamP,
\end{split}
\end{align*}
where we have used the fact that $\delta _n \leq \veps_n$.
\end{proof}

%
%

\subsection{Smoothening step in continuous setting}
\label{Sec:Smooth}

The purpose of this section is to present some basic results that will enable us to regularize solutions to the continuity equation without increasing their total action. We do this via a standard mollification argument.

Let $\{ \bH_s \}_{s >0}$ be the heat semigroup on $\T^d$. For an arbitrary $\mu \in \mathcal{P}(\T^d)$, the convolution $\bH_s ( \mu)$ is the measure whose density $\rho_s$ with respect to the Lebesgue measure is given by
\[ \rho_s(x) :=  \int_{\T^d} J_s( y-x ) d \mu(y) = J_s \ast \mu(x), \]
where in the above and in the remainder, $y-x$ has to be interpreted modulo $\Z^d$, and $J_s$ denotes the heat kernel

\begin{equation} 
J_s(u):=   \frac{1}{(4 \pi s)^{d/2}} \exp\left(  -\frac{\lvert u\rvert^2}{4s}  \right).
\label{GaussianJ} 
\end{equation}

Let $t \in [0,1] \mapsto (\mu_t, \vec{V}_t)$ be a solution to the continuity equation and let $s>0$. Consider $\mu_t^s:= \bH_s(\mu_t)$ and let $\vec{V}_t^s$ be the density of the vector valued measure $\bH_s(\vec{V}_t \mu_t) := J_s \ast (\vec{V}_t \mu_t)$. Then, the curve $t \in [0,1] \mapsto (\mu_t^s, \vec{V}_t^s)$ satisfies the continuity equation in flux form:
\[  \frac{d}{dt} \mu_t^s + \divergence(\vec{V}_t^s) =0.  \]
The proof of this statement can be found in \cite{GigliBook}. For the convenience of the reader, we point out that the proof relies on the fact that $\bH_s$ is a convolution operator and thus, for arbitrary  $\phi \in C^\infty(\T^d)$,
\begin{align}
\begin{split}
\int_{\T^d} \vec{V}_t^h(x) \cdot \nabla \phi(x) dx = & \int_{\T^d}\int_{\T^d}J_s(z-x) \vec{V}_t(z) \cdot \nabla \phi (x) d\mu_t(z) dx
\\ &= \int_{\T^d}\left( \int_{\T^d}  J_s(z-x) \nabla \phi(x) dx    \right) \cdot \vec{V}_t(z) d\mu_t(z)
\\& =\int_{\T^d} (J_s \ast \nabla \phi)(z) \cdot \vec{V}_t(z)d\mu_t(z)
\\&= \int_{\T^d} \vec{V}_t(z) \cdot \nabla ( J_s \ast \phi)(z) d\mu_t(z),
\end{split}
\label{ConvoluVectors}
\end{align}
where the last equality follows from the fact that $\nabla_z J_s(z-x) = - \nabla_x J_s(z-x) $ which guarantees that $ (J_s \ast \nabla \phi)(z) = \nabla (J_s \ast \phi) (z)$.  On the other hand, at the level of the measures, we have
\[ \frac{d}{dt} \int_{\T^d} \phi(x) d\mu_t^s(x) = \frac{d}{dt}\int_{\T^d} (J_s\ast\phi)(x) d \mu_t(x), \]
which together with \eqref{ConvoluVectors} implies that $t \in [0,1] \mapsto (\mu_t^s, \vec{V}_t^s)$ is indeed a solution to the continuity equation in flux form.

We have seen that by convolving a solution to the continuity equation with the heat kernel we obtain a more regular solution (i.e. smoother densities and vector fields). In the next two propositions we summarize some facts that we need in the remainder.

\begin{proposition}(Heat flow on $\T^d$). The following assertions hold for every $s>0$: 
	\begin{enumerate}
		\item There exists constants $c_1(s)>0$ and $C_2(s) < \infty$ such that for every $\mu \in \mathcal{P}(\T^d)$ the density $\rho_s$
		of $\bH_s(\mu)$ satisfies
		\[ \rho_s \geq c_1(s) \quad \text{ and } \Lip(\rho_s)\leq C_2(s). \]
		
		\item For every $\mu \in \mathcal{P}(\T^d)$ we have
		\[ W(\mu, \bH_s(\mu)) \leq \sqrt{2ds} .\]  
		\item There exists a constant $C_3(s) < \infty$  such that for any $f \in L^1(\T^d)$ we have
		\[ \lVert \bH_s(f) \rVert_\infty + \Lip(\bH_s(f)) + \lVert  D^2 \bH_s(f) \rVert_{\infty} \leq C_3(s) \lVert f \rVert_{L^1(\T^d)}  \]
		\item Let $t\in [0,1] \mapsto (\mu_t, \vec{V}_t)$ be a geodesic in $\mathcal{P}(\T^d)$ (w.r.t. the Wasserstein distance). Let $\rho_{s,t}$ be the density of $\bH_s(\mu_t)$ and $\vec{V}_{s,t}$ be the density of the vector valued measure $\bH_s(\vec{V}_t \mu_t)$. Then, the curve $t \in [0,1] \mapsto (\rho_{s,t}, \vec{V}_{s,t})$ is a solution to the continuity equation in flux form and we have
		\[ \int_{0}^1 \int_{\T^d} \frac{\lvert\vec{V}_{s,t}(x) \rvert^2}{\rho_{s,t}(x)} dx dt \leq W(\mu_0, \mu_1))^2.  \]
		
	\end{enumerate}
	\label{HeatFlowTorus} 
\end{proposition}
\begin{proof}
All of the above assertions and their proofs are presented in \cite{gigmaas}. We notice that the second assertion follows from Lemma 7.1.10 in \cite{GigliBook} by noticing that the Wasserstein distance between $\mu$ and $\bH_s \ast \mu$ seen as measures on the torus is smaller than the Wasserstein distance of $\mu$ and $\bH_s \ast \mu$ seen as measures on $[0,1]^d$; the value of the constant on the right hand side of the inequality was obtained using well known identities for a Gaussian. 
\end{proof}

\begin{proposition}
Let $t \in [0,1] \mapsto (\mu_t, \vec{V}_t)$ be a solution to the continuity equation and for $s>0$ let $\vec{V}_t^s$ be the density of the vector valued measure $\bH_s ( \vec{V}_t \mu_t)$. Then, there exists a constant $C_4(s)$ such that 
\[  \int_{0}^1 \left( \lVert \vec{V}_t^s \rVert_\infty   + \Lip(\vec{V}_t^s) + \lVert D^2 \vec{V}_t^s \rVert_\infty  \right) dt \leq  C_4(s) \left( \int_{0}^{1} \int_{\T^d}  \lvert  \vec{V}_t(x) \rvert^2 d \mu_t(x)  dt \right)^{1/2}.  \]
\label{HeatFlowTorus2}	
\end{proposition}	
\begin{proof}
As in (3) in Proposition \ref{HeatFlowTorus} (from the definition of convolution) we can obtain that
\begin{equation*}
\int_{0}^1 \left( \lVert \vec{V}_t^s \rVert_\infty   + \Lip(\vec{V}_t^s) + \lVert D^2 \vec{V}_t^s \rVert_\infty  \right) dt \leq  C_4(s) \int_{0}^{1} \int_{\T^d} \lvert \vec{V}_t(x) \rvert d \mu_t(x) dt.
\end{equation*}
A simple application of Jensen's inequality implies the desired result.
\end{proof}

\subsection{Smoothening step in  discrete setting}
\label{Sec:SmoothDisc}

The main results in the previous section stated that starting from an arbitrary solution $t \in [0,1] \mapsto (\mu_t, \vec{V}_t)$ to the continuity equation, we could construct a more regular solution whose total action is smaller than that of the original curve. In this section we present analogue results for the discrete case. Let us start by introducing some notation that we will use in the remainder of this section.

Let $T_n$ be the optimal transportation map, in the $\infty$-OT sense, between $\nu$ and $\nu_n$.  We introduce a map 
\[P_n : \mathcal{P}(X_n) \rightarrow \mathcal{P}(\T^d)\]
that associates to every discrete density $\rho_n \in \mathcal{P}(X_n)$ a density $\rho(x) $ (with respect to $\nu$) satisfying
\[ \rho(x):=\rho_n(\x_i) , \quad x \in U_i,\]
where $U_i = T_n^{-1}(\x_i)$. In particular, the density $\rho$ is constant in each of the sets $U_i$. Notice that this construction was already used in Lemma \ref{DiscreteW2L8}.

\begin{proposition}
Suppose that $ t\in [0,1] \mapsto (\rho_{n,t}, V_{n,t} )$ is a solution to the discrete continuity equation. Then, there exists a constant $c>0$ such that for all $ b \in [ c\delta_n, i_d]$ ( i.e. $b$ smaller than the injectivity radius of $\T^d$, but no smaller than $c\delta_n$) there exists a curve $t \in [0,1] \mapsto (\rho_{n,t}^b, V_{n,t}^b)$ satisfying the following:

\begin{enumerate}
\item (Continuity equation) For every $i=1, \dots,n$,
\begin{equation*}
\frac{d}{dt}\rho_{n,t}^b(\x_i) + \divergence_{n, \tilde{\veps}_n}(V_{n,t}^b)(\x_i) =0 
\end{equation*}
where
\[ \divergence_{n , \tilde{\veps}_n}(V_{n,t}^b)(\x_i) := \frac{1}{n \tilde \veps_n}\sum_{j=1}^n (V_{n,t}^b(\x_i, \x_j)- V_{n,t}^b(\x_j, \x_i) ) w_{\tilde{\veps}_n}(\x_i, \x_j)    \]
and	
\[\tilde{\veps}_n := \veps_n + C_d \delta_n, \]
for a constant $C_d$ only depending on dimension.
\item (Action is almost decreased) The action, 
\begin{equation*}
 \mathcal{A}_{n,\tilde{\veps}_n}(\rho_{n,t}^b, V_{n,t}^b) := \frac{1}{n^2}\sum_{i,j} \frac{(V_{n,t}^b(\x_i, \x_j))^2}{\theta(\rho_{n,t}^b(\x_i), \rho_{n,t}^b(\x_j))} w_{\tilde{\veps}_n}(\x_i, \x_j)   
\end{equation*}
satisfies
\[ \mathcal{A}_{n,\tilde{\veps}_n}(\rho_{n,t}^b, V_{n,t}^b)  \leq (1+ C_d\frac{\delta_n}{\veps_n} ) \mathcal{A}_n(\rho_{n,t}, V_{n,t}) , \quad \forall t\in [0,1]. \]

\item (Regularity of curves) For every $i,j$ we have
\[  \lvert \rho_{n,t}^b(\x_i) - \rho_{n,t}^b (\x_j) \rvert  \leq  C_7(b) \left( \lvert \x_i- \x_j \rvert + 2 C_d \delta_n \right), \] 
for a constant $C_7(b)$ that depends on $b$.
\end{enumerate}

Moreover, if $ M \geq \rho_{n,t} \geq a >0$ for every $t \in [0,1]$, then the same lower and upper bounds hold for $\rho_{n,t}^b$ and in addition  
\[  W(P_n(\rho_{n,0}), P_n(\rho_{n,0}^b) ) \leq C\left(M + \frac{1}{a} \right) \Lip(\rho_{n,0}) ( b +  \delta_n ),   \]
\[ W(P_n(\rho_{n,1}) , P_n(\rho^b_{n,1}) ) \leq C \left(M + \frac{1}{a} \right) \Lip(\rho_{n,1}) ( b +  \delta_n )  ,\]
where $M \geq \rho_{n,0},   \rho_{n,1} \geq a>0$. 

\label{SmoothDiscrete}
\end{proposition}	
Let us stop for a moment and point out the main technical difficulty to establish this proposition. We recall that the smoothening step in the continuous setting was obtained by a simple mollification of probability measures and vector fields. These operations respected the continuity equation, and because they were actually convolutions, Jensen's inequality guaranteed the desired comparison for the total actions. In the geometric graph setting however, there is no direct analogue to such procedure, partly because there is no obvious way to talk about coordinates of discrete vector fields.  Our idea to go around this issue is as follows. We appropriately match the point cloud $X_n$ with a set of points $Y_B$ aligned on a regular grid on $\T^d$; we then conveniently map $\mathcal{P}(X_n)$ and $\mathfrak{X}(X_n)$ into $\mathcal{P}(Y_B)$ and $\mathfrak{X}(Y_B)$. We exploit the regular grid structure of $Y_B$ to define a convolution for discrete densities and vector fields on $Y_B$. The idea is then to we map our original curve in $\mathcal{P}(\T^d)$ to a curve in $\mathcal{P}(Y_B)$ which is then regularized. After we have completed the regularization in $\mathcal{P}(Y_B)$, we map back to $\mathcal{P}(X_n)$. The main work is in making sure that the inequalities relating total actions are maintained when mapping back and forth between $\mathcal{P}(X_n)$ and $\mathcal{P}(Y_B)$.

%

\begin{proof}[Proof of Proposition \ref{SmoothDiscrete} ]

Let $B$ be the largest natural number below $n$ of the form $B= \kappa^d$ for some natural number $\kappa$. We consider $Y_B$ a set of $B$ points in $\T^d$ forming a regular grid and denote by $\hat{\nu}_B$ the empirical measure associated to the points in $Y_B$. It is straightforward to check that there is a transportation map $\hat{T}_B : \T^d \rightarrow Y_B$ satisfying 
\[\hat{T}_{B \sharp} \nu = \hat{\nu}_B, \quad \lVert \hat{T}_B - Id \rVert_\infty \leq \frac{C_d}{B^{1/d}},  \]
From this, we conclude that
\[ d_\infty(\hat{\nu}_B, \nu)  \leq \frac{C_d}{B^{1/d}} \leq  \frac{C_d}{n^{1/d}} \leq C_d \delta_n,\]
where the last inequality follows from Remark \ref{Remarkinfty}. Now, recall that $d_\infty(\nu, \nu_n) = \delta_n$. It follows from the triangle inequality that
\[ d_\infty(\nu_n , \hat{\nu}_B)  \leq C_d \delta_n. \]
 In particular, there exists a transportation \textit{plan} $\pi_n : X_n \times Y_B \rightarrow [0,1]$ between $\nu_n$ and $\hat{\nu}_B$ satisfying
\[ \lvert \x_i - y_k \rvert \leq C_d \delta_n \text{ for all } \x_i \in X_n \text{ and } y_k \in Y_B  \text{ with } \pi_n(\x_i, y_k)>0.  \]
As a transport plan, we have
\[ \sum_{k=1}^B \pi_n(\x_i, y_k) = \frac{1}{n}, \quad \forall i=1, \dots, n,   \]
and
\[ \sum_{i=1}^n \pi_n(\x_i, y_k) = \frac{1}{B}, \quad \forall k=1, \dots, B.   \]

For an arbitrary discrete density $\rho_n \in \mathcal{P}(X_n)$, we define an associated discrete density $\hat{\rho}_B \in \mathcal{P}(Y_B)$ (the set $\mathcal{P}(Y_B)$ is defined in the same way we defined $\mathcal{P}(X_n)$ using $\hat{\nu}_B$ as reference measure) by
\begin{equation}
	\hat{\rho}_B(y_k):= B \sum_{i=1}^n \pi_n(\x_i, y_k) \rho_n(\x_i), \quad k=1, \dots, B.
\end{equation}
Likewise, for a given discrete vector field $V_n \in \mathfrak{X}(X_n)$, we define an associated vector field $\hat{V}_B \in \mathfrak{X}(Y_B)$ by 
\[ \hat{V}_B(y_k, y_l ) := \begin{cases} 
B^2 \sum_{i=1}^n \sum_{j=1}^n \pi_n(\x_i, y_k) \pi_n(\x_j, y_l)   V_n(\x_i, \x_j) \frac{w_{\veps_n}(\x_i, \x_j)}{ w_{\hat{\veps}_B}(y_k , y_l)}, & \text{ if } w_{\hat{\veps}_B}(y_k , y_l) >0
\\ 0 , & \text{otherwise},  
\end{cases} \]
where in the above 
\[ \hat{\veps}_B:= \veps_n + 2 C_d \delta_n. \]
Notice that from the definition of $\hat{\veps}_B$ we conclude that if $\pi_n(\x_i, y_k)>0$, $\pi_n(\x_j , y_l)>0$ and $w_{\hat{\veps}_B}(y_k, y_l ) = 0$, then $w_{\veps_n}(\x_i, \x_j ) = 0$ as well. Because of this we may write, 
\begin{equation}
\hat{V}_B(y_k, y_l ) = B^2 \sum_{i=1}^n \sum_{j=1}^n \pi_n(\x_i, y_k) \pi_n(\x_j, y_l)   V_n(\x_i, \x_j) \frac{w_{\veps_n}(\x_i, \x_j)}{w_{\hat{\veps}_B}(y_k, y_l )},
\label{auxDiscreteSmooth}
\end{equation}
with the convention that $\frac{0}{0}=0$. under this convention, we can multiply both sides of the above identity by $w_{\hat{\veps}_B}(y_k, y_l)$ and cancel out the $w_{\hat{\veps}_B}(y_k, y_l)$ term appearing in the denominator of the right hand side. It follows that,
\begin{align}
\begin{split}
\divergence_B(\hat{V}_B)(y_k) &:=  \frac{1}{ B  \veps_n} \sum_{l=1}^B \left(\hat{V}_B(y_k, y_l) - \hat{V}_B(y_l , y_k)\right) w_{\hat{\veps}_B}(y_k , y_l) 
 \\ &= \frac{B}{\veps_n} \sum_{l=1}^B \sum_{i=1}^n \sum_{j=1}^n \pi_n(\x_i, y_k) \pi_n(\x_j, y_l)( V_n(\x_i, \x_j) - V_n(\x_j, \x_i) ) w_{\veps_n}(\x_i, \x_j) 
 \\  &= \frac{B}{n \veps_n} \sum_{i=1}^n  \left(\sum_{j=1}^n (V_n(\x_i, \x_j) - V_n(\x_j, \x_i))  w_{\veps_n}(\x_i, \x_j)\right) \pi_n(\x_i, y_k)
 \\ &= \sum_{i=1}^n B \pi_n(\x_i, y_k) \divergence_n(V_n)(\x_i). 
\end{split}
\end{align}
Thus, if $t \in [0,1] \mapsto (\rho_{n,t}, V_{n,t})$ satisfies the discrete continuity equation, we have
\begin{align}
\begin{split}
\frac{d}{dt} \hat{\rho}_{B,t}(y_k)  &= \sum_{i=1}^n B \pi_n(\x_i, y_k) \frac{d}{dt}\rho_{n,t}(\x_i) 
\\& =- \sum_{i=1}^n B \pi_n(\x_i, y_k) \divergence_n(V_{n,t})(\x_i)
\\& =-  \divergence_B(\hat{V}_B )(y_k), \quad \forall k=1, \dots, B.
\end{split}
\end{align}
Furthermore, the convexity of the function $(u,w,v) \mapsto \frac{v^2}{\theta(u,w)} $ (which follows from (5) in Assumption \ref{assump}) and Jensen's inequality implies that
\begin{align}
\begin{split}
\mathcal{A}_B (\hat{\rho}_{B,t}, \hat{V}_{B,t}) := & \frac{1}{B^2} \sum_{k=1}^B \sum_{l=1}^B \frac{(\hat{V}_{B,t}(y_k , y_l))^2 }{ \theta(\hat{\rho}_{B,t}(y_k) , \hat{\rho}_{B,t}(y_l)  )}   w_{\hat{\veps}_B}(y_k , y_l)  
\\& \leq \sum_{k,l} \sum_{i,j}  \pi_n(\x_i,y_k)\pi_n(\x_j,y_l) \frac{(V_{n,t}(\x_i , \x_j))^2 }{ \theta(\rho_{n,t}(\x_i) , \rho_{n,t}(\x_j)  )}   \frac{w_{\veps_n}(\x_i , \x_j)}{w_{\hat{\veps}_B}(y_k , y_l)} w_{\veps_n}(\x_i , \x_j)
\\& \leq \left(\frac{\hat{\veps}_B}{\veps_n}\right)^d \sum_{k,l} \sum_{i,j}  \pi_n(\x_i,y_k)\pi_n(\x_j,y_l) \frac{(V_{n,t}(\x_i , \x_j))^2 }{ \theta(\rho_{n,t}(\x_i) , \rho_{n,t}(\x_j)  )}  w_{\veps_n}(\x_i , \x_j)
\\&= \frac{1}{n^2}\left(\frac{\hat{\veps}_B}{\veps_n}\right)^d  \sum_{i,j}  \frac{(V_{n,t}(\x_i , \x_j))^2 }{ \theta(\rho_{n,t}(\x_i) , \rho_{n,t}(\x_j)  )}  w_{\veps_n}(\x_i , \x_j)
\\&= \left(\frac{\hat{\veps}_B}{\veps_n}\right)^d \mathcal{A}_n(\rho_{n,t}, V_{n,t} )
\\&\leq (1+ C_d\frac{\delta_n}{\veps_n}) \mathcal{A}_n(\rho_{n,t}, V_{n,t} ).
\end{split}
\label{ActionGrid}
\end{align}

We now exploit the regular grid structure of $Y_B$ to define mollified versions of discrete densities in $\mathcal{P}(X_n)$ and mollified versions of discrete vector fields in $\mathfrak{X}(X_n)$. The smoothening operation that we introduce respects the discrete continuity equation, provided we  enlarge $\veps_n$ a bit. Let $G : \R^d \rightarrow [0, \infty)$  be a radial standard mollifier. That is, let $G$ be a radial smooth function with support in $B(0,1)$ which integrates to $1$. In what follows we identify $G$ with its radial profile. For $b< i_d$ (recall $i_d$ is the injectivity radius of $\T^d$) we define
\[ G_b(\lvert x -y \rvert) := \frac{1}{B \gamma_b b^d}G \left(\frac{\lvert x-y \rvert}{b} \right), \]
where $\gamma_b$ is a constant introduced so as to guarantee that
\[ \sum_{k=1}^B G_b(\lvert y_k- y_m \rvert) =1, \quad \forall m=1, \dots, B.\]
Notice that the map $m \mapsto \sum_{k=1}^B G_b(\lvert y_k- y_m \rvert)$ is indeed constant due to the regular grid structure of $Y_B$.  Moreover, we can estimate $\gamma_b$ easily, noticing that 
\begin{align*}
\begin{split}
 \gamma_b &= \frac{1}{B}\sum_{k=1}^B   \frac{1}{b^d} G\left( \frac{\lvert y_k - y_m \rvert}{b}  \right) = \int_{\T^d}\frac{1}{b^d} G\left( \frac{\lvert \hat{T}_B(x) - y_m \rvert}{b}  \right)dx 
 \\&= \int_{\T^d}\frac{1}{b^d} G\left( \frac{\lvert x - y_m \rvert}{b}  \right)dx  + O\left(\frac{1}{B^{1/d} b }\right) = 1 + O\left(\frac{1}{B^{1/d} b }\right).
 \end{split} 
 \end{align*}
In particular, 
\[   \lvert  \gamma_b - 1  \rvert  \leq C\frac{\delta_n}{b},\]
and we can take $c=2C$ in the statement of the proposition to guarantee that  $1/2 \leq \gamma_b \leq 2$; this is all that we need to know about $\gamma_b$ in what follows.
\nc

Given $\rho_n \in \mathcal{P}(X_n)$ let 
\[ \rho_n^b (\x_i):= n \sum_{k=1}^B \sum_{m=1}^B  \pi_n(\x_i, y_k) G_b(\lvert y_m - y_k \rvert) \hat{\rho}_B(y_m), \]
and for a given $V_n \in \mathfrak{X}(X_n)$ let $V_n^b \in \mathfrak{X}(X_n)$ be 
\[ V_n^b(\x_i, \x_j) :=n^2\frac{\tilde{\veps}_n}{\veps_n} \sum_{k=1}^B \sum_{l=1}^B \sum_{m=1}^B  \hat{V}_B(y_m , y_m + (y_l - y_k)) G_b(\lvert y_m - y_k \rvert) \pi_n(\x_i, y_k) \pi_n(\x_j, y_l) \frac{w_{\hat{\veps}_B}(y_m, y_m+(y_l - y_k))}{w_{\tilde{\veps}_n}(\x_i, \x_j) }, \] 
where 
\[\tilde{\veps}_n:= \hat{\veps}_B + 2C_d\delta_n.\]
In the above we follow the same convention introduced when we defined $\hat{V}_B$.

As when we defined the vector field $\hat{V}_B$, we make the same observations regarding how to deal with the above expression in case $w_{\tilde{\veps}_n}(\x_i, \x_j) =0$. Notice also that $y_m+  (y_l - y_k)$ in the above formula indeed belongs to $Y_B$ so that the expression $\hat{V}_B(y_m , y_m + (y_l - y_k))$ makes sense. Moreover, notice that for fixed $m$ and $k$, the map $l \mapsto y_m+  (y_l - y_k)$ is a parameterization of $Y_B$. In particular, we can write
\begin{align*}
\begin{split}
\sum_{j=1}^n V_n^b(\x_i, \x_j) w_{\tilde{\veps}_n}(\x_i, \x_j)  &= \frac{n^2 \tilde{\veps}_n}{B\veps_n}  \sum_{k=1}^B  \sum_{m=1}^B \sum_{l=1}^B \hat{V}_B(y_m , y_l) G_b(\lvert y_m - y_k \rvert) \pi_n(\x_i, y_k)  w_{\hat{\veps}_B}(y_m, y_l).
\end{split} 
\end{align*}
In addition, if we swap the roles of $\x_i$ and $\x_j$ and use the regular grid structure of $Y_B$, we can write after a change of variables
\[ V_n^b(\x_j, \x_i)   = n^2\frac{\tilde{\veps}_n}{\veps_n}\sum_{k=1}^B \sum_{l=1}^B \sum_{m=1}^B \hat{V}_B(y_m + (y_l- y_k) , y_m )G_b(\lvert y_m - y_k   \rvert) \pi_n(\x_i, y_k) \pi_n(\x_j, y_l) \frac{w_{\hat{\veps}_B}(y_m + (y_l- y_k) , y_m )}{w_{\tilde{\veps}_n}(\x_i, \x_j)},\]
and from this obtain
\[ \sum_{j=1}^n V_n^b(\x_j, \x_i) w_{\tilde{\veps}_n}(\x_i, \x_j)  =  \frac{n^2 \tilde{\veps}_n}{B\veps_n} \sum_{k=1}^B  \sum_{m=1}^B \sum_{l=1}^B \hat{V}_B(y_l, y_m) G_b(\lvert y_m - y_k \rvert) \pi_n(\x_i, y_k)  w_{\hat{\veps}_B}(y_m, y_l). \]
In particular, it follows that
\begin{align}
\begin{split}
\divergence_{n, \tilde{\veps}_n}(V_n^b)(\x_i):=& \frac{1}{n \tilde{\veps}_n}\sum_{j=1}^n (V_n^b(\x_i, \x_j)- V_n^b(\x_j, \x_i) ) w_{\tilde{\veps}_n}(\x_i, \x_j)  
\\&= n \sum_{k=1}^B \sum_{m=1}^B \pi_n(\x_i, y_k)G_b(\lvert y_m - y_k \rvert) \divergence_B(\hat{V}_B)(y_m).
\end{split}
\end{align}
We conclude that if  $t \in [0,1] \mapsto (\rho_{n,t}, V_{n,t}) $ is a solution to the discrete continuity equation, then 
\begin{align}
\begin{split}
\frac{d}{dt}\rho_{n,t}^b (\x_i) & =  \sum_{k=1}^B \sum_{m=1}^B n \pi_n(\x_i, y_k) G_b(\lvert y_m - y_k \rvert) \frac{d}{dt}\hat{\rho}_{B,t}(y_m) 
\\ & = - \sum_{k=1}^B \sum_{m=1}^B n \pi_n(\x_i, y_k) G_b(\lvert y_m - y_k \rvert) \divergence_B(\hat{V}_{B,t})(y_m)
\\& = - \divergence_{n , \tilde{\veps}_n}(V_n^b)(\x_i),
\end{split}
\end{align}
proving in this way the first part of the proposition.

\red
\nc

To estimate the action of the curve $t \mapsto (\rho_{n,t}^b , V_{n,t}^b)$ fix $i,j$ and notice that 
\[ \rho_{n,t}^b(\x_i) = \sum_{k=1}^{B} \sum_{l=1}^B \sum_{m=1}^B n^2 \pi_n(\x_i, y_k)\pi_n(\x_j, y_l) G_b(\lvert y_m - y_k  \rvert) \hat{\rho}_B(y_m) \]
and also for every $k=1, \dots, B$,
\[  \rho_{n,t}^b(\x_j) =  \sum_{l=1}^B \sum_{m=1}^B n \pi_n(\x_j, y_l) G_b(\lvert y_m - y_k  \rvert) \hat{\rho}_B(y_m+(y_l-y_k)),  \]
which in particular implies that
\[ \rho_{n,t}^b(\x_j) = \sum_{k=1}^{B} \sum_{l=1}^B \sum_{m=1}^B n^2 \pi_n(\x_i, y_k)\pi_n(\x_j, y_l) G_b(\lvert y_m - y_k  \rvert) \hat{\rho}_B(y_m+(y_l - y_k)). \]
From these identities for $\rho_{n,t}^b(\x_i), \rho_{n,t}^b (\x_j)$, the definition of $V_{n,t}^b(\x_i, \x_j)$, the convexity of the function $(u,v,w)\mapsto \frac{v^2}{\theta(u,w)}$, Jensen's inequality, and after similar computations as those in \eqref{ActionGrid}, it follows that
\begin{align}
\begin{split}
 \mathcal{A}_{n,\tilde{\veps}_n}(\rho_{n,t}^b, V_{n,t}^b) &:= \frac{1}{n^2}\sum_{i,j} \frac{(V_{n,t}^b(\x_i, \x_j))^2}{\theta(\rho_{n,t}^b(\x_i), \rho_{n,t}^b(\x_j))} w_{\tilde{\veps}_n}(\x_i, \x_j)      
\\ & \leq \left(\frac{\tilde{\veps}_n}{\hat{\veps}_B} \right)^d\left(\frac{\tilde{\veps}_n}{\veps_n}\right)^2 \mathcal{A}_B(\hat{\rho}_{B,t},\hat{V}_{B,t} ).
\end{split}
\end{align}
Combining this inequality with \eqref{ActionGrid}, we obtain 
\begin{equation}
\mathcal{A}_{n,\tilde{\veps}_n}(\rho_{n,t}^b, V_{n,t}^b) \leq (1 + C_d \frac{\delta_n}{\veps_n}) \mathcal{A}_n(\rho_{n,t}, V_{n,t}),
\end{equation}
establishing in this way the second part of the proposition.

Regarding the regularity of the densities $\rho_{n,t}^b$ we notice that from the definitions it follows that
\begin{align*}
\begin{split}
\lvert \rho_{n,t}^b(\x_j) - \rho_{n,t}^b(\x_i) \rvert   & \leq \sum_{k=1}^B \sum_{l=1}^B \sum_{m=1}^B n^2 \pi_n(\x_j, y_l) \pi_n(\x_i, y_k) \left \lvert G_b(\lvert y_l - y_m \rvert) - G_b(\lvert y_k - y_m \rvert) \right \rvert \hat{\rho}_{B,t}(y_m)
\\& \leq C_7(b) \left( \lvert \x_i- \x_j \rvert + 2 C_d \delta_n \right),
\end{split}
\end{align*}
where $C_7(b)$ is a constant only depending on $b$.

To show the last part of the proposition let us assume that $\rho_{n,t} \geq a >0$ for every $t \in [0,1]$. Notice that from the definitions of $\hat{\rho}_{B,t}$ and $\rho_{n,t}^b$ it follows that $\rho_{n,t}^b \geq a >0$. On the other hand, for every $i=1, \dots, n$ 
\begin{align*}
\begin{split}
 \lvert   \rho_{n,0}(\x_i) - \rho_{n,0}^b(\x_i) \rvert & =  \left \lvert \sum_{k=1}^B \sum_{m=1}^B \sum_{j=1}^n n B \pi_n(\x_i, y_k) \pi_n(\x_j, y_m) G_b(\lvert y_m - y_k \rvert) ( \rho_{n,0}(\x_i) -  \rho_{n,0}(\x_j) )  \right \rvert
 \\& \leq  \Lip(\rho_{n,0}) \sum_{k=1}^B \sum_{m=1}^B \sum_{j=1}^n n B \pi_n(\x_i, y_k) \pi_n(\x_j, y_m) G_b(\lvert y_m - y_k \rvert) \lvert \x_i -  \x_j\rvert   
\\& \leq  \Lip(\rho_{n,0}) ( b + 2C_d \delta_n   ) 
\end{split}
\end{align*}
where the last inequality follows from the fact that $G_b$ vanishes if its argument is greater than $b$. In other words, we have shown  
\[ \lVert   \rho_{n,0} - \rho_{n,0}^b \rVert_{L^\infty(\nu_n)}  \leq \Lip(\rho_{n,0}) ( b + 2C_d \delta_n ) .\] 
Now, let $\rho_0$ and $\rho_0^b$ be the densities of $P_n(\rho_{n,0})$ and $P_n(\rho_{n,0}^b)$ with respect to the Lebesgue measure. We observe that the lower and upper bounds for $\rho_{n,0}$ and $\rho_{n,0}^b$ are also bounds for $\rho_0$ and $\rho_0^b$. Moreover, $\lVert \rho_0 - \rho_0^b  \rVert_{L^\infty(\T^d)}= \lVert \rho_{n,0} - \rho_{n,0}^b \rVert_{L^\infty(\nu_n)} $. From these facts and Lemma \ref{W2L8} we conclude that 
\[ W(P_n(\rho_{n,0}), P_n(\rho_{n,0}^b) ) \leq  C\left(M_0 + \frac{1}{a} \right) \Lip(\rho_{n,0}) ( b +  \delta_n ). \]
In exactly the same way, we obtain 
\[    W_n(P_n(\rho_{n,1}), P_n(\rho_{n,1}^b) )  \leq   C \left(M_1 + \frac{1}{a} \right) \Lip(\rho_{n,1}) ( b +  \delta_n ).  \]
\nc
\end{proof}

\
The previous proposition allows us to regularize a given curve $t \mapsto (\rho_{n,t}, V_{n,t})$. Next we modify a given curve so as to obtain densities that are bounded away from zero. To achieve this we first need to introduce a discrete version of (global) Poincare inequality.
\begin{lemma}(Poincare inequality in discrete setting).
For all $f_n \in L^2(\nu_n)$ with  $\langle f_n , \mathds{1} \rangle_{L^2(\nu_n)}=0$ we have
\[ \lVert f_n \rVert^2_{L^2(\nu_n)} \leq \frac{1}{\lambda_n} D_n(f_n)  \leq C_{d} D_n(f_n), \]
where $\lambda_n$ is the first non-trivial eigenvalue of the Graph Laplacian $\Delta_n$ and $D_n$ is the discrete \textit{Dirichlet energy} defined by
\[ D_n(f_n):= \langle \nabla_n f_n , \nabla_n f_n\rangle_{\mathfrak{X}(X_n)}. \]
\label{PoincareLemma}
\end{lemma}
\begin{proof}
	
This result is obtained using standard tools from the literature of elliptic PDEs (in this case in the graph setting), noticing that the first non-trivial eigenvalue of the graph Laplacian $\lambda_n$ is $C_{d}(\veps_n + \frac{\delta_n}{\veps_n})$-close to  the first non-trivial eigenvalue of the Laplacian on the torus as it follows from the main result in \cite{BIK} (see also \cite{trillos2018spectral}).   

\end{proof}

\begin{proposition}
	Suppose that $ t \in [0,1]  \mapsto (\rho_{n,t}, V_{n,t} )$ is a solution to the discrete continuity equation and that
	\[\rho_{n,0} \geq m_0 >0, \quad \rho_{n,1} \geq m_1 >0 .\]
	For every $a \in (0,1) $ define
	\[  \rho_{n,t}^a := (1-a)\rho_{n,t} + a  \]
	and
	\[ V_{n,t}^a:= (1-a)V_{n,t}.  \]
	Then, $t \mapsto (\rho_{n,t}^a , V_{n,t}^a)$ solves the discrete continuity equation, 
	\[  \mathcal{A}_{n}(\rho_{n,t}^a, V_{n,t}^a ) \leq (1-a) \mathcal{A}_{n}(\rho_{n,t}, V_{n,t}) , \quad \forall t \in [0,1], \]
	and $\rho_{n,t}^a \geq a$ for every $t$. Moreover, 
	\[ W(P_n(\rho_{n,0}^a) , P_n(\rho_{n,0})) \leq  C_{d} \Lip(\rho_{n,0})\frac{a}{\sqrt{m_0}} , \quad   W(P_n(\rho_{n,1}^a) , P_n(\rho_{n,1})) \leq C_{d} \Lip(\rho_{n,1})\frac{a}{\sqrt{m_1}} .\]
\label{a}
\end{proposition}

\begin{proof}
The fact that $t  \in [0,1] \mapsto (\rho_{n,t}^a, V_{n,t}^a)$ solves the discrete continuity equation follows directly from the definitions. The relation between the actions of the two curves follows from the definitions, and the monotonicity and homogeneity of $\theta$. To prove the last inequalities let us denote by $\rho_0$ the density of $P_n(\rho_{n,0})$ and by $\rho_{0}^a$ the density of $P_n(\rho_{n,0}^a)$. It follows that,
\[ \rho_{0}^a = (1-a) \rho_{0} +a. \]
We can then define $\rho^s:= \rho_0 +sa(1- \rho_0 ) $ for $s \in [0,1]$ and note that $\rho^0=\rho_0$ and $\rho^1= \rho_0^a$. Let $\vec{V}_s$ be the (constant in time) vector field 
\[ \vec{V}_s :=-a\nabla \left( \Delta^{-1}(1- \rho_{0}) \right), \quad s \in [0,1],\]
where $\Delta^{-1}$ is the pseudoinverse of $\Delta$ (the Laplacian on $\T^d$), which is well defined for $L^2$-functions with average zero. It is straightforward to see that $s \in [0,1] \mapsto (\rho^s , \vec{V}_s) $ satisfies the continuity equation in flux form. Moreover, using standard $L^2$ elliptic estimates we deduce that
\[ (W(P_n(\rho_0), P_n(\rho_0^a)))^2 \leq \int_{0}^1 \mathcal{A}(\rho^s, \vec{V}_s )ds \leq C_d\frac{a^2}{m_0}\int_{\T^d}\lvert \nabla \left( \Delta^{-1}(1- \rho_{0}) \right) \rvert^2dx \leq  C_d\frac{a^2}{m_0} \lVert 1 - \rho_0  \rVert^2_{L^2(\nu)}.  \] 
Noticing that $\lVert 1 - \rho_0  \rVert^2_{L^2(\nu)} = \lVert 1 - \rho_{n,0}  \rVert^2_{L^2(\nu_n)}$, we can use the discrete Poincare inequality (Lemma \ref{PoincareLemma}) to conclude that
\[  (W(P_n(\rho_0), P_n(\rho_0^a)))^2 \leq C_{d} \frac{a^2}{m_0} D_n(\rho_{n,0}) \leq C_{d} \frac{a^2}{m_0} (\Lip(\rho_{n,0}))^2,  \]
where the last inequality follows directly from the definition of $D_n(\rho_{n,0})$. We can bound $W(P_n(\rho_1), P_n(\rho_1^a))$ in exactly the same way.
\end{proof}

\subsection{Auxiliary results}
\label{Sec:aux}

In this section we present some identities that will be used in the proof of Theorem \ref{thm:main}. In these identities we connect local with non-local quantities. 
\begin{lemma}
Let $\vec{V}$ be a $C^1$-vector field defined on $\R^d$. Then, for every $x \in \R^d$
\[ \frac{1}{\veps} \int_{\R^d} \vec{V}(y) \cdot \frac{(y-x)}{\veps} w_\veps(x,y) dy = \alpha_d \int_{0}^\infty g(x,\veps r) r^{d+1}\eta(r) dr ,\]
where 
\[ g(x,s) := \fint_{B(x,s)}\divergence(\vec{V}) dy , \]
and $\alpha_d$ is the volume of the unit ball in $\R^d$.
\label{LemmaAux1}
\end{lemma}
\begin{proof}
Using polar coordinates we obtain
\begin{align}
\begin{split}
\frac{1}{\veps} \int_{\R^d} \vec{V}(y) \cdot \frac{(y-x)}{\veps} w_\veps(x,y) dy &= \int_{0}^\infty \left(  \int_{\partial B(0,r)} \vec{V}(x+ \xi) \cdot \frac{\xi}{\veps}  dS(\xi)  \right) \frac{1}{\veps^d} \eta\left( \frac{r}{\veps} \right)  \frac{dr}{\veps} \\
 &= \int_{0}^\infty \left(  \int_{\partial B(0,r)} \vec{V}(x+ \xi) \cdot \frac{\xi}{r}  dS(\xi)  \right) \frac{1}{\veps^d}\frac{r}{\veps} \eta\left( \frac{r}{\veps} \right)  \frac{dr}{\veps} \\
 &= \int_{0}^\infty \left(  \int_{ B(x,r)} \divergence(\vec{V})  dy  \right) \frac{1}{\veps^d}\frac{r}{\veps} \eta\left( \frac{r}{\veps} \right)  \frac{dr}{\veps} \\
 &= \alpha_d \int_{0}^\infty \left(  \fint_{ B(x,r)} \divergence(\vec{V})  dy  \right) \frac{r^{d+1}}{\veps^{d+1}} \eta\left( \frac{r}{\veps} \right)  \frac{dr}{\veps}, \\
 \end{split}
 \end{align}
where the third equality follows from the divergence theorem. The desired identity is obtained from the above expression after changing variables $\hat{r}:= \frac{r}{\veps}$. 
\end{proof}

\begin{remark}
Notice that since $\eta$ is zero for $r \geq 1$, the vector field $\vec{V}$ only has to be defined on an $\veps$-ball around $x$ for the identity in the previous lemma to make sense.
\end{remark}

\begin{lemma}
Let $\vec{u}$ be a vector in $\R^d$. Then, for every $x \in \R^d$
\[  \int_{\R^d} \left(\vec{u}\cdot \frac{y-x}{\veps} \right)  \frac{y-x}{\veps} w_{\veps}(x,y) dy = \alpha_d \sigma_\eta \vec{u}. \]
\label{LemmaAux212}
\end{lemma}
\begin{proof}
Using polar coordinates we obtain
\begin{align}
\begin{split}
  \int_{\R^d} \left( \vec{u}\cdot \frac{y-x}{\veps}  \right) \frac{y-x}{\veps} w_{\veps}(x,y) dy & = \frac{1}{\veps^d}\int_{0}^\infty  \left( \int_{\partial B(0, r)} \vec{u}\cdot \frac{\xi}{\veps}  \frac{\xi}{\veps} dS(\xi) \right)  \eta \left( \frac{r}{\veps} \right)dr 
\\& = \left( \int_{0}^\infty \eta \left( \frac{r}{\veps} \right) \frac{r^{d+1}}{\veps^{d+1}} \frac{dr}{\veps} \right) \int_{\partial B(0, 1)} (\vec{u}\cdot \xi)  \xi dS(\xi)
\\&= \sigma_\eta \int_{\partial B(0, 1)} (\vec{u}\cdot \xi)  \xi dS(\xi)
\end{split}
\end{align}
Now, by symmetry it is clear that for every $\vec{u} \in \R^d$,
\[ \int_{\partial B(0, 1)} (\vec{u}\cdot \xi)  \xi dS(\xi) = C \vec{u}, \]
for some constant $C$ to be determined. Let $\{e_1, \dots, e_d \}$ be the canonical basis for $\R^d$. Then,
\[ \int_{\partial B(0, 1)} (e_i\cdot \xi)^2 dS(\xi) = C , \quad \forall i =1, \dots, d. \] 
Hence,
\[   d \alpha_d = \int_{\partial B(0,1)} \sum_{i=1}^d  (e_i\cdot \xi)^2 dS(\xi) = d C ,\]
where $\alpha_d$ is the volume of the unit ball in $\R^d$ (note that $d\alpha_d$ is the surface area of the unit sphere). We conclude that $C= \alpha_d$ and the result now follows. 
\end{proof}

The next is an immediate consequence of the previous result.

\begin{corollary}
Let $\vec{u}\in \R^d$ . Then, 
\[  \int_{\R^d} \left(  \frac{(y-x)\cdot \vec{u}}{\veps}  \right)^2 w_\veps(x,y) dy = \alpha_d \sigma_\eta \lvert \vec{u} \rvert^2 , \quad \forall x \in \R^d,\]
where $\alpha_d$ is the volume of the unit ball in $\R^d$.
\label{LemmaAux3}
\end{corollary}
%
%

\begin{lemma}
Let $V : \R^d \times \R^d \rightarrow \R$ be square integrable with respect to the measure $w_\veps(x,y) dy dx$ and define the vector field $\vec{V}$ by
\[  \vec{V}(x):= \int_{\R^d} V(x,y) w_\veps(x,y)  \frac{(y-x)}{\veps} dy, \quad x \in \R^d.\]
Then, for almost every $x$,
\[  \lvert \vec{V}(x) \rvert ^2 \leq   \alpha_d \sigma_\eta  \int_{\R^d}  ( V(x,y))^2  w_\veps(x,y) dy. \]
\label{LemmaAux2}
\end{lemma}
\begin{proof}
For a given $x \in \R^d$, there exists a unit vector $e \in \R^d$ for which 
\[  \lvert  \vec{V}(x)\rvert = \vec{V}(x) \cdot e.\]
Indeed if $\vec{V}(x) \not = 0 $ we may take $e = \frac{\vec{V}(x)}{\lvert  \vec{V}(x)\rvert}$ and if $\vec{V}(x)=0$ we may take any unit vector.  Then, 
\begin{align}
\begin{split}
\lvert  \vec{V}(x) \rvert ^2 & = (\vec{V}(x) \cdot e)^2 \\
&=\left(  \int_{\R^d} V(x,y) w_\veps(x,y) \frac{(y-x) \cdot e}{\veps} dy \right)^2\\
&\leq    \left( \int_{\R^d}\left(  \frac{(y-x) \cdot e}{\veps}\right)^2  w_\veps(x,y) dy \right) \left( \int_{\R^d}(V(x,y))^2w_\veps(x,y)   dy \right)
\label{LemmaAux2Idaux}
\end{split}
\end{align}
where the inequality follows from Cauchy-Schwartz inequality. Combining this with Corollary \ref{LemmaAux3} we obtain the desired result.
\end{proof}

\section{Proof of Theorem \ref{thm:main}}
\label{Sec:Proof}

At the beginning of Subsection \ref{Sec:SmoothDisc} we introduced a map $P_n : \mathcal{P}(X_n) \rightarrow \mathcal{P}(\T^d)$ using the optimal transport map $T_n$.  We now introduce a map  $Q_n : \mathcal{P}(\T^d) \rightarrow \mathcal{P}(X_n)$, $P_n'$s adjoint, which associates to every probability measure $\mu \in \mathcal{P}(\T^d)$ a discrete density $\rho_n$ given by
\[ \rho_n(\x_i):= \frac{\mu(U_i)}{\nu(U_i)} = n \mu(U_i), \quad  i=1, \dots, n.\] 
For $\mu \in \mathcal{P}(\T^d)$ with a density $\rho$ with respect to $\nu$, we may abuse notation slightly and write $Q_n(\rho)$ instead of $Q_n(\mu)$. Notice that in that case we can write
\[  Q_n(\rho)(\x_i)= \fint_{U_i} \rho(x) dx , \quad i=1, \dots, n.  \]

\begin{remark}
From the definitions it is clear that the map $P_n\circ Q_n: \mathcal{P}(X_n)\rightarrow \mathcal{P}(X_n) $ is the identity.
\end{remark}


To prove Theorem \ref{thm:main}, we need to construct a map $F_n : \mathcal{P}(\T^d) \rightarrow  \mathcal{P}(X_n)$ whose distortion is small and is such that every element in $\mathcal{P}(X_n)$ is close, in the $W_n$ sense, to an element in $F_n(\mathcal{P}(\T^d))$. The map $F_n$ that we construct takes the form 
\[ F_n :=  Q_n \circ \bH_s, \] 
for a conveniently chosen value of $s>0$, where we recall $\bH_s$ is the heat flow. We split our proof into three parts that can be summarized as follows. Part 1: lower bound for $W(\mu_0, \mu_1)$, where $\mu_0, \mu_1 \in \mathcal{P}(\T^d)$ are arbitrary measures. Part 2: upper bound for $W(\mu_0, \mu_1)$. In the final part we wrap up the argument. 

In what follows, $C$ will be used to denote a constant that depends on $d$ and $ \theta$ only; the value of $C$ may change from line to line. We will also consider constants that depend on some specific parameters. For example if the parameter is denoted with $h$, we will write $C(h)$ to represent a constant that depends on $h$; the value of $C(h)$ may change from line to line.

\subsection{Part 1}

Let $\mu_0, \mu_1 \in \mathcal{P}(\T^d)$ be arbitrary and let $t \mapsto (\mu_t, \vec{V}_t)$ be a geodesic connecting $\mu_0$ and $\mu_1$, so that in particular 
\[ \int_{0}^1 \int_{\T^d} \lvert \vec{V}_t(x) \rvert^2 d\mu_t(x) dt  = (W(\mu_0, \mu_1))^2. \]
For a fixed value of $s>0$ (to be chosen later on) we define $\rho_t^s:= \bH_s( \mu_t) $ and $\vec{V}_t^s:= \bH_s(\vec{V}_t \mu_t)$. We recall that  Proposition \ref{HeatFlowTorus} implies that $t \mapsto (\rho_t^s, \vec{V}_t^s)$ is a solution to the continuity equation in flux form.

\nc

Let $\rho_{n,t}:=Q_n(\rho_t^s)$. For every $i=1, \dots, n$ we have
\begin{equation}
 \frac{d}{dt} \rho_{n,t} (\x_i) = \fint_{U_i} \frac{d}{dt}\rho_t^s(x) dx = - \fint_{U_i}  \divergence (\vec{V}_t^s)d x , \quad t \in(0,1) .
 \label{lowerboundConteqn1}
\end{equation}
We claim that 
\begin{equation} 
\label{lowerbound}
W_n(Q_n(\bH_s(\mu_0)), Q_n(\bH_s(\mu_1))) \leq W(\mu_0, \mu_1) + C(s) \left( \veps_n + \frac{\delta_n}{\veps_n} \right)^{1/2}, 
\end{equation} 
where $C(s)$ blows up as $s\rightarrow 0$. Our plan to prove \eqref{lowerbound} is the following. First, notice that the curve $t \in [0,1] \mapsto \rho_{n,t}$ in $\mathcal{P}(X_n)$ connects the discrete measures $Q_n(\rho_0^s)$ and $Q_n(\rho_1^s)$, but \eqref{lowerboundConteqn1} does not have the form of the discrete continuity equation.  Because of this, we construct a suitable solution $t \in [0,1] \mapsto (\tilde{\rho}_{n,t}, V_{n,t})$ to the discrete continuity equation which starts at $Q_n(\rho_0^s)$ and stays close to the curve $t \in [0,1] \mapsto \rho_{n,t}$; crucially, its discrete total action is comparable to that of the curve $t \in [0,1] \mapsto (\rho_t^s, \vec{V}_t^s)$ . From the triangle inequality we will be able to conclude that up to a small error $W_n(Q_n(\rho_0^s), Q_n(\rho_1^s))$ is below $W(\rho_0, \rho_1)$.

With this road map in mind, let us start by defining $\tilde{\veps}_n := \veps_n - 2 \delta_n$ and let $V_{n,t}$ be the vector field
\[ V_{n,t}(\x_i,\x_j) :=\frac{1}{\alpha_d \sigma_\eta} \fint_{U_i}\fint_{U_j} \vec{V}_t^s(x)\cdot \frac{y-x}{\tilde{\veps}_n} \frac{w_{\tilde{\veps}_n}(x,y)}{w_{\veps_n}(\x_i, \x_j)} dy dx, \quad i,j=1, \dots, n, \]
where we recall the $U_i$ are the transport cells induced by the $\infty$-OT map between $\nu$ and $\nu_n$. In the above, as has been done routinely throughout the paper, we use the convention that $\frac{0}{0}=0$.
Let $\tilde{\rho}_{n,t}$ be the solution to the equation 
\begin{equation}
\frac{d}{dt} \tilde{\rho}_{n,t}(\x_i) + \divergence_n( V_{n,t})(\x_i) =0 , \quad \forall t\in (0,1) , \quad \forall i=1, \dots,n,
\end{equation}
with initial condition $\tilde{\rho}_{n,0}= \rho_{n,0}$. We will later prove that $\tilde{\rho}_{n,t}$ is indeed a discrete density by showing that it is a non-negative function (see Remark \ref{RemarkNonNegative}). This however will be a direct consequence of the fact that $\tilde{\rho}_{n,t}$ is uniformly close to $\rho_{n,t}$. In what follows we focus on proving this.


First, notice that 
\begin{align}
\begin{split}
  \alpha_d \sigma_\eta \divergence_n(V_{n,t})(\x_i)  &=  \frac{1}{  n \veps_n}   \sum_{j=1}^n   \fint_{U_i} \fint_{U_j}   \vec{V}_t^s(x) \cdot \frac{y-x}{\tilde \veps_n}  w_{\tilde \veps_n}(x,y)dy dx 
 \\& +  \frac{1}{  n \veps_n}   \sum_{j=1}^n   \fint_{U_i} \fint_{U_j}   \vec{V}_t^s(y) \cdot \frac{y-x}{ \tilde \veps_n}  w_{\tilde \veps_n}(x,y)dy dx.  
  \end{split} 
  \label{divnPart1}
  \end{align}
Now, for every  $x \in U_i$  
\[  \frac{1}{n}\sum_{j=1}^n   \fint_{U_j}   \frac{y-x}{\tilde \veps_n}  w_{\tilde \veps_n}(x,y)dy   =   \int_{\T^d}  \frac{y-x}{\tilde \veps_n}  w_{\tilde \veps_n}(x,y)dy = 0, \]
 where the last equality is due to radial symmetry of the kernel used to define the weights. We conclude that the first term on the right hand side of \eqref{divnPart1} is equal to zero. Thanks to Lemma \ref{LemmaAux1}, the second term on the right hand side of \eqref{divnPart1} can be written as 
 \[  \frac{1}{\veps_n}\fint_{U_i}  \left(\int_{\T^d} \vec{V}_t^s(y) \cdot \frac{y-x}{\tilde \veps_n} w_{\tilde \veps_n}(x,y) dy \right) dx =\alpha_d \frac{\tilde{\veps}_n}{\veps_n}  \fint_{U_i} \int_{0}^\infty   g(x,\tilde \veps_n r)  r^{d+1} \eta(r) dr dx, \]
 where we recall that
\[ g(x,\tilde{ \veps}_n r) = \fint_{B(x, \tilde \veps_n r)}  \divergence(\vec{V}_t^s)(y) dy.   \]
Therefore,
\begin{equation*}
\divergence_n(V_{n,t})(\x_i)  = \frac{1}{\sigma_\eta}\frac{\tilde{\veps}_n}{\veps_n} \fint_{U_i} \int_{0}^\infty   g(x,\tilde \veps_n r)  r^{d+1} \eta(r) dr dx, \quad \forall i=1, \dots, n,
\end{equation*}
and in particular,  for every $i=1, \dots, n$,
\begin{align*}
\begin{split}
\lvert \divergence_n(V_{n,t})(\x_i)  -  \frac{\tilde{\veps}_n}{\veps_n} \fint_{U_i} \divergence(\vec{V}_t^s) dx  \rvert & =  \frac{\tilde{\veps}_n}{\veps_n} \left \lvert\frac{1}{\sigma_\eta} \fint_{U_i} \int_{0}^\infty         ( g(x, \tilde \veps_n r)  - \divergence \vec{V}_t^s(x) ) r^{d+1} \eta(r)dr dx \right \rvert  
\\ & = \frac{\tilde{\veps}_n}{\veps_n}  \left \lvert\frac{1}{\sigma_\eta} \fint_{U_i} \int_{0}^\infty     \fint_{B(x,\tilde \veps_n r)}  ( \divergence{\vec{V}_t^s}(y) - \divergence(\vec{V}_t^s)(x) ) r^{d+1} \eta(r) dydr dx \right \rvert
\\& \leq  \frac{ \lVert D^2 \vec{V}_t^s \rVert_\infty}{\sigma_\eta}   \veps_n   \fint_{U_i} \int_{0}^\infty r^{d+2}\eta(r)  dr dx
\\&= C_{d} \lVert D^2 \vec{V}_t^s  \rVert_\infty  \veps_n .
\end{split}
\end{align*}
On the other hand, for every $i=1, \dots, n$
\[ \lvert \left( \frac{\tilde{\veps}_n}{\veps_n} - 1  \right)   \fint_{U_i} \divergence(\vec{V}_t^s) dx  \rvert \leq 2  \Lip(\vec{V}_t^s)  \frac{\delta_n}{\veps_n}.  \]
Combining the previous two inequalities we deduce that for every $t \in [0,1]$  and for every $i=1, \dots, n$,

\begin{align}
\begin{split}
  \lvert  \tilde{\rho}_{n,t}(\x_i) - \rho_{n, t}(\x_i)  \rvert & \leq  C ( \veps_n + \frac{\delta_n}{\veps_n} ) \int_{0}^1 (\Lip(\vec{V}_r^s) + \lVert D^2 \vec{V}_r^s  \rVert_\infty ) dr
  \\& \leq    C(s) ( \veps_n + \frac{\delta_n}{\veps_n} )  \left( \int_{0}^1 \int_{\T^d} \lvert  \vec{V}_r(x) \rvert^2 d \mu_r(x)  dr \right)^{1/2}  
  \\& =  C(s) ( \veps_n + \frac{\delta_n}{\veps_n} )  W(\mu_0,\mu_1)
  \\& \leq C(s) ( \veps_n + \frac{\delta_n}{\veps_n} ) ,
  \end{split}
  \label{Boundrhotilderho}
\end{align}
\nc
where the second inequality follows from Proposition \ref{HeatFlowTorus2}, and the latter one from the fact that the Wasserstein distance between two arbitrary measures in $\mathcal{P}(\T^d)$ is bounded above by the diameter of $\T^d$.

From Proposition \ref{HeatFlowTorus} we know that  
\[\rho_{n,t}(\x_i) = \fint_{U_i} \rho_t^s(x) dx \geq c_1(s)>0 , \quad \forall i=1, \dots, n. \]
Assuming we choose $s>0$ so that
\begin{equation}
C(s) (\veps_n + \frac{\delta_n}{\veps_n}) \leq \frac{c_1(s)}{2},   
\label{ChoiceN1}
\end{equation}
we conclude from \eqref{Boundrhotilderho} that for all $t \in [0,1]$, both  $\tilde{\rho}_{t,n}$ and $\rho_{t,n}$ are lower bounded by $c_1(s)/2$ (in particular we will later send $n \rightarrow \infty$ first, then we will send $s \rightarrow 0$). In particular, Lemma \ref{DiscreteW2L8} implies that
\begin{equation}
W_n(\rho_{n,1}, \tilde{\rho}_{n,1}) \leq  C(s) ( \veps_n + \frac{\delta_n}{\veps_n} )  + C \veps_n \diamP .
\label{AuxPart11}
\end{equation}

To estimate the total action of $t \mapsto (\tilde{\rho}_{n,t}, V_{n,t})$, first notice that for every $i=1,\dots, n$ and every $x \in U_i$ we have

\begin{align}
\begin{split}
\lvert  \tilde{\rho}_{n,t}(\x_i) - \rho_t^s(x)   \rvert & \leq \lvert  \tilde{\rho}_{n,t}(\x_i) - \rho_{n,t}(\x_i) \rvert  +   \lvert \rho_{n,t}(\x_i) - \rho_t^s(x)   \rvert     \\
& \leq  C(s)  \left(   \veps_n + \frac{\delta_n}{\veps_n}   \right) +2 \delta_n C_2(s)=: \kappa(n,s), 
\end{split}
\label{AuxPart100}
\end{align}
\nc
which follows from \eqref{Boundrhotilderho}, the definition of $\rho_{n,t}(\x_i)$, and Proposition \ref{HeatFlowTorus}. From the properties of $\theta$, it follows that for all $i,j=1,\dots, n$ and all $x\in U_i$, $y \in U_j$, 
\begin{equation}
( 1+ \frac{2\kappa(n,s)}{c_1(s)}  ) \theta( \tilde{\rho}_{n,t}(\x_i), \tilde{\rho}_{n,t}(\x_j)   ) \geq \theta( \tilde{\rho}_{n,t}(\x_i) + \kappa(n,s), \tilde{\rho}_{n,t}(\x_j)  + \kappa(n,s)  ) \geq \theta (\rho_t^s(x), \rho_t^s(y)  ),    
\label{AuxPart1001}
\end{equation}
where the last inequality follows from the monotonicity of $\theta$ and \eqref{AuxPart100}. Likewise, 
\begin{equation}
\rho_t^s(x) =\theta( \rho_t^s(x) , \rho_t^s(x ) ) \leq  \left( 1+    \frac{C_2(s)}{c_1(s)} \lvert y-x \rvert \right)  \theta( \rho_t^s(x), \rho_t^s(y) )  , \quad \forall x, y \in \T^d.
\label{AuxPart1002}
\end{equation}
Therefore, 
\begin{align}
\begin{split}
 &\mathcal{A}_n(\tilde{\rho}_{n,t}, V_{n,t})=  \frac{1}{\alpha_d^2 \sigma_\eta^2 n^2}\sum_{i,j}   \left( \fint_{U_i} \fint_{U_j}  \vec{V}_t^s(x) \cdot \frac{y-x}{\tilde{\veps}_n} \frac{w_{\tilde{\veps}_n }(x,y)}{w_{\veps_n}(\x_i, \x_j)}dy dx \right)^2  \frac{w_{\veps_n}(\x_i, \x_j)}{\theta( \tilde{\rho}_{n,t}(\x_i), \tilde{\rho}_{n,t}(\x_j)  )}
\\& \leq  \frac{1}{\alpha_d^2 \sigma_\eta^2 n^2}(1 + C_d \frac{\delta_n}{\veps_n}) \sum_{i,j}    \fint_{U_i} \fint_{U_j}  \left(\vec{V}_t^s(x) \cdot \frac{y-x}{\tilde{\veps}_n}\right)^2  w_{\tilde{\veps}_n}(x,y)dy dx \frac{1}{\theta( \tilde{\rho}_{n,t}(\x_i), \tilde{\rho}_{n,t}(\x_j)  )}
\\&  \leq  \frac{1}{\alpha_d^2 \sigma_\eta^2 n^2}(1 + C_d \frac{\delta_n}{\veps_n})( 1+  \frac{2\kappa(n,s)}{c_1(s)} ) \sum_{i,j}    \fint_{U_i} \fint_{U_j}  \left(\vec{V}_t^s(x) \cdot \frac{y-x}{\tilde{\veps}_n}\right)^2  \frac{w_{\tilde{\veps}_n}(x,y)}{\rho_t^s(x)} (1 + C(s) \lvert y-x \rvert )dy dx 
\\&= \frac{1}{\alpha_d^2 \sigma_\eta^2 }(1 + C_d \frac{\delta_n}{\veps_n})( 1+  \frac{2\kappa(n,s)}{c_1(s)} ) \int_{\T^d} \int_{\T^d}  \left(\vec{V}_t^s(x) \cdot \frac{y-x}{\tilde{\veps}_n}\right)^2  \frac{w_{\tilde{\veps}_n}(x,y)}{\rho_t^s(x)} (1 + C(s) \lvert y-x \rvert )dy dx
\\& \leq \frac{1}{\alpha_d \sigma_\eta } \left( 1 + C \frac{\delta_n}{\veps_n} + C  \frac{\kappa(n,s)}{c_1(s)}  + C(s) \veps_n \right) \int_{\T^d} \frac{\lvert \vec{V}_t^s (x) \rvert^2}{\rho^s_{t}(x)} dx.
\end{split}
\end{align}
In the above, the first inequality was obtained using Jensen's inequality and the fact that  $\frac{w_{\tilde{\veps}_n}(x,y)}{w_{\veps_n}(\x_i, \x_j)} \leq  \left(\frac{\veps_n}{\tilde{\veps}_n}\right)^d \leq 1 + C_d \frac{\delta_n}{\veps_n}$; the second inequality was obtained using \eqref{AuxPart1001} and \eqref{AuxPart1002}; the third inequality follows from Corollary \ref{LemmaAux3}. Integrating the above inequality over $t$ and using Proposition \ref{HeatFlowTorus}, we deduce that
\begin{align*}
\begin{split}
  (W_n(\tilde{\rho}_{n,0} , \tilde{\rho}_{n,1}))^2 & \leq \int_{0}^1 \mathcal{A}_n(\tilde{\rho}_{n,t}, V_{n,t}) dt \leq \frac{1}{\alpha_d \sigma_\eta} (W(\mu_0, \mu_1 ))^2 +  C \left(\frac{\delta_n}{\veps_n} +    \frac{\kappa(n,s)}{c_1(s)} + C(s) \veps_n \right).  
 \end{split}  
\end{align*}
 Together with \eqref{AuxPart11} and the triangle inequality, the above implies that 
\begin{align}
\begin{split}
 &W_n(\rho_{n,0}, \rho_{n,1})  \leq W_n(\rho_{n,0}, \tilde{\rho}_{n,1}) + W_{n}( \tilde{\rho}_{n,1}, \rho_{n,1}) 
 \\&\leq\frac{1}{\sqrt{\alpha_d \sigma_\eta}} W(\mu_0, \mu_1 ) +  C \left(  \frac{\delta_n}{\veps_n} + \frac{\kappa(n,s)}{c_1(s)} + C(s)\veps_n \right)^{1/2} + C(s) \left( \veps_n +  \frac{\delta_n}{\veps_n}\right) + C \veps_n \diamP.
\end{split}
\label{Main1}
\end{align}
\nc

\subsection{Part 2}
For $\mu_0, \mu_1 \in \mathcal{P}(\T^d)$ let $\rho_{n,0}:= Q_n(\bH_s(\mu_0))$ and $\rho_{n,1}:= Q_n(\bH_s(\mu_1))$ as in Part 1. Notice that by the definition of $Q$ and Proposition \ref{HeatFlowTorus} we have
\[ C_dC_2(s) \geq \rho_{n,0}, \rho_{n,1} \geq c_1(s) , \quad \Lip(\rho_{n,0}), \Lip(\rho_{n,1}) \leq C_2(s). \]

Let $t \in [0,1] \mapsto (\rho_{n,t}, V_{n,t})$ be a geodesic connecting $\rho_{n,0}$ and $\rho_{n,1}$. 
By the discrete smoothening step in Section \ref{Sec:SmoothDisc} (Proposition \ref{SmoothDiscrete} and Proposition \ref{a}), for fixed $a\in(0,1)$ and $b \in (c \delta_n , i_d)$, we can construct a curve $t \in [0,1] \mapsto (\tilde{\rho}_{n,t}, \tilde{V}_{n,t})$
satisfying the following properties:
\begin{enumerate}
	\item For every $i,j=1,\dots, n$ 
	\[\lvert \tilde{\rho}_{n,t}(\x_i) - \tilde{\rho}_{n,t}(\x_j) \rvert \leq C(b) (\lvert \x_i - \x_j \rvert +  \delta_n ).\]
	\item For all $t \in[0,1]$, 
	\[C(b) \geq \tilde{\rho}_{n,t} \geq a . \] 
	\item For all $i=1, \dots, n$, 
	\[ \frac{d}{dt}\tilde{\rho}_{n,t}(\x_i) + \divergence_{n , \tilde{\veps}_n}(\tilde{V}_{n,t})(\x_i) =0.  \]
	\item $\int_{0}^1 \mathcal{A}_{n , \tilde{\veps}_n}(\tilde{\rho}_{n,t}, \tilde{V}_{n,t})dt \leq (1+ C_d \frac{\delta_n}{\veps_n})(W_n(\rho_{n,0}, \rho_{n,1})  )^2 $.
	\item $W(P_n(\tilde{\rho}_{n,j}), P_n(\rho_{n,j})) \leq C(s) a  + C(s) (1+ \frac{1}{a} ) (  b+ \delta_n)  $, for $j=0,1$.
\end{enumerate}
Here $\tilde{\veps}_n= \veps_n + C_d \delta_n$ and $\divergence_{n, \tilde{\veps}_n}$, $\mathcal{A}_{n, \tilde{\veps}_n}$ are defined in Proposition \ref{SmoothDiscrete}.

For every $t \in [0,1]$, let $\tilde{\rho}_t:= P_n(\tilde{\rho}_{n,t})$. Observe that for every $i=1, \dots, n$ we have
\[  \frac{d}{dt}\tilde{\rho}_t(x)   =    - \divergence_{n , \tilde{\veps}_n}(\tilde{V}_{n,t})(\x_i) , \quad \forall x \in U_i  .\]
We notice that although the  path $t  \in [0,1]\mapsto \tilde{\rho}_t$ in $\mathcal{P}(\T^d)$ does indeed connect the measures with densities $\tilde{\rho}_0$ and $\tilde{\rho}_1$, the equation satisfied by $\tilde{\rho}_t$ is not written in the form of the continuity equation. To go around this, we construct a suitable solution $t \in [0,1] \mapsto (\tilde \rho_{t}^h, V_{t}^h)$ to the continuity equation which starts at $\tilde \rho_0^h$ (a mollified version of $\tilde{\rho}_0$ using a parameter $h>0$ to be chosen later on) and stays close to the original curve $t \in [0,1] \mapsto \tilde{\rho}_{t}$. 
The smoothness of $\tilde{\rho}_t^h$ and the regularity of $\tilde{\rho}_t$ (coming from the regularity of $\tilde{\rho}_{n,t}$) will guarantee that the $L^\infty$ norm of $\tilde{\rho}_t^h - \tilde{\rho}_t$ is small. We use this fact to establish that the action associated to the curve $t \in [0,1] \mapsto (\tilde{\rho}_t^h, \vec{V}_t^h)$ is, up to a small error, smaller than the discrete action of the curve $t \in [0,1] \mapsto (\tilde{\rho}_{n,t}, \tilde{V}_{n,t})$.
\nc

Let $\vec{V}_{t}$ be the vector field defined by
\[ \vec{V}_t(x) := \int_{\T^d} \tilde{V}_{n,t}(T_n(x),T_n( y)) w_{\tilde{\veps}_n}(T_n(x),T_n(y)) \frac{y-x}{\tilde{\veps}_n} dy, \quad x \in \T^d, \]
and for fixed $h>0$ (to be chosen later on) consider the mollified vector field
\[  \vec{V}_t^h:= J_h \ast \vec{V}_t,   \]
where we recall $J_h$ is de Gaussian defined in \eqref{GaussianJ}. Let $\phi \in C^\infty(\T^d)$ and let $\psi := J_h \ast \phi$. Then, as in \eqref{ConvoluVectors},
\begin{align*}
\begin{split}
\int_{\T^d}\vec{V}_t^h(x) \cdot \nabla \phi(x) dx &=\int_{\T^d} \vec{V}_t(x) \cdot \nabla \psi(x) dx 
\\&= \sum_{i,j}  \tilde V_{n,t}(\x_i,\x_j) w_{\tilde{\veps}_n}(\x_i,\x_j) \int_{U_i} \int_{U_j} \frac{y-x}{\tilde{\veps}_n} \cdot \nabla \psi(x) dy dx
\\&  =: \sum_{i,j} \tilde V_{n,t}(\x_i,\x_j) w_{\tilde{\veps}_n}(\x_i,\x_j) \int_{U_i} \int_{U_j} \frac{\psi(y)-\psi(x)}{\tilde{\veps}_n}  dy  dx + \gamma_{n,t}
\\&= \frac{1}{n^2}\sum_{i,j}\tilde V_{n,t}(\x_i,\x_j) w_{\tilde\veps_n}(\x_i,\x_j) \fint_{U_i} \fint_{U_j} \frac{\psi(y)-\psi(x)}{\tilde\veps_n}  dy  dx + \gamma_{n,t}
\\&=: \frac{1}{n^2}\sum_{i,j} \tilde V_{n,t}(\x_i,\x_j) w_{\tilde\veps_n}(\x_i,\x_j)  \frac{\psi(\x_j)-\psi(\x_i)}{\tilde\veps_n}   + \gamma_{n,t} +\beta_{n,t}
\\&= \langle \tilde V_{n,t}, \nabla_{n, \tilde{\veps}_n} \psi \rangle_{\mathfrak{X}(X_n)} +\gamma_{n,t} + \beta_{n,t}.
\end{split}
\end{align*}
Let us estimate $\gamma_{n,t}$ and $\beta_{n,t}$. Using Cauchy-Schwartz inequality, it is straightforward to show that
\begin{align*}
\begin{split}
\lvert \gamma_{n,t}  \rvert & \leq \sum_{i,j} \lvert \tilde V_{n,t}(\x_i, \x_j) \rvert w_{\tilde\veps_n}(\x_i, \x_j)\int_{U_i}\int_{U_j} \frac{\lVert D^2 \psi \rVert_\infty \lvert y-x \rvert^2}{2 \tilde \veps_n}dy dx
\\&\leq   \lVert  D^2 \psi \rVert_\infty   \sum_{i,j} \lvert \tilde V_{n,t}(\x_i, \x_j) \rvert w_{\tilde \veps_n}(\x_i, \x_j)\int_{U_i}\int_{U_j}\frac{\lvert y-x \rvert^2}{ \tilde \veps_n}dxdy
\\& \leq   \lVert  D^2 \psi \rVert_\infty \tilde{\veps}_n  \sum_{i,j} \lvert \tilde V_{n,t}(\x_i, \x_j) \rvert w_{\tilde \veps_n}(\x_i, \x_j)\int_{U_i}\int_{U_j}\frac{\lvert y-x \rvert^2}{ \tilde \veps_n^2}dxdy
\\& \leq  C \lVert D^2 \psi \rVert_\infty \veps_n  (\mathcal{A}_{n,\tilde{\veps}_n}(\tilde{\rho}_{n,t}, \tilde{V}_{n,t} ))^{1/2},
\end{split}
\end{align*}
%
and that
\begin{align*}
\begin{split}
\lvert \beta_{n,t} \rvert  & \leq   \frac{1}{n^2} \sum_{i,j} \lvert   \tilde V_{n,t}(\x_i, \x_j)\rvert  w_{\tilde \veps_n}(\x_i, \x_j) \frac{2 \Lip(\psi) \delta_n}{\tilde \veps_n}
\\& \leq C\Lip(\psi)\frac{\delta_n}{\veps_n}(\mathcal{A}_{n, \tilde{\veps}_n}(\tilde{\rho}_{n,t}, \tilde{V}_{n,t}))^{1/2}.
\end{split}
\end{align*}
Therefore, for every $t \in [0,1]$,
\begin{align}
\begin{split}
\left \lvert   \int_{\T^d}  \vec{V}_t ^h \cdot \nabla \phi dx - \langle \tilde{V}_{n,t}, \nabla_{n, \tilde{\veps}_n} \psi \rangle_{\mathfrak{X}(X_n)}    \right \rvert &\leq C( \Lip(\psi) + \lVert D^2 \psi \rVert_\infty )(\veps_n + \frac{\delta_n}{\veps_n} ) (\mathcal{A}_{n, , \tilde \veps_n}(\tilde{\rho}_{n,t}, \tilde{V}_{n,t}))^{1/2} 
\\& \leq   C(h)  \nc \lVert \phi \rVert_{L^1(\T^d)} (\veps_n + \frac{\delta_n}{\veps_n} ) (\mathcal{A}_{n, \tilde \veps_n}(\tilde{\rho}_{n,t}, \tilde{V}_{n,t}))^{1/2}, 
\end{split}
\label{AuxIneq0}
\end{align}
where the last inequality follows from Proposition \ref{HeatFlowTorus}.

Let $\tilde{\rho}_{t}^h$ be the solution of
\[  \frac{d}{dt}\tilde{\rho}_t^h (x)  +  \divergence(\vec{V}_t^h)(x) =0 \]
with initial condition $\tilde \rho_{0}^h= J_h \ast \tilde{\rho}_0$. We will show that for all $t \in (0,1)$, $\tilde{\rho}_t^h$ is indeed a density by showing that it is a non-negative function. This, however, will follow from the fact that $\tilde \rho_{t}^h$ is uniformly close to $\tilde{\rho}_t$; we focus on showing this. 

Indeed, at the discrete level,
\[ \frac{d}{dt}  \langle  \tilde{\rho}_{n,t} , J_h \ast \phi \rangle_{L^2(\nu_n)}   =    \langle \tilde V_{n,t} ,  \nabla _{n, \tilde{\veps}_n} (J_h \ast \phi) \rangle_{\mathfrak{X}(X_n)},   \]
and at the continuum level
\[ \frac{d}{dt} \langle  \tilde{\rho}_{t}^h ,  \phi \rangle_{L^2(\nu)}  = \int_{\T^d} \vec{V}_t^h \cdot \nabla \phi(x) dx. \]
Using \eqref{AuxIneq0} we deduce that
\[ \left \lvert \frac{d}{dt} \left(   \langle \tilde{\rho}_t^h, \phi \rangle_{L^2(\nu)} -\langle  \tilde{\rho}_{n,t} , J_h \ast \phi \rangle_{L^2(\nu_n)}\right) \right  \rvert  \leq     C(h) \nc \lVert \phi \rVert_{L^1(\T^d)} (\veps_n + \frac{\delta_n}{\veps_n} ) (\mathcal{A}_{n,\tilde{\veps}_n}(\tilde{\rho}_{n,t}, \tilde{V}_{n,t}))^{1/2} . \]
From this and Jensen's inequality it follows that for all $t \in [0,1]$,
\begin{align*}
\begin{split}
\lvert  \langle   \tilde{\rho}_t^h, \phi \rangle_{L^2(\nu)}  -  \langle  \tilde{ \rho}_{n,t}, J_h \ast \phi \rangle_{L^2(\nu_n)}   \rvert & \leq   C(h) \nc \lVert \phi \rVert_{L^1(\T^d)}  \left(   \int_{0}^{1} \mathcal{A}_{n, \tilde{\veps}_n}(\tilde \rho_{n,r}, \tilde V_{n,r})dr  \right)^{1/2}(\veps_n + \frac{\delta_n}{\veps_n} ) 
\\&+ \lvert \langle \tilde{\rho}_0^h, \phi \rangle_{L^2(\nu)}  -  \langle  \tilde{\rho}_{n,0}, J_h \ast \phi \rangle_{L^2(\nu_n)}     \rvert
\\&\leq C(h) \nc\lVert \phi \rVert_{L^1(\T^d)} (\veps_n + \frac{\delta_n}{\veps_n} ) +  \lvert \langle \tilde{\rho}_0^h, \phi \rangle_{L^2(\nu)}  -  \langle  \tilde{\rho}_{n,0}, J_h \ast \phi \rangle_{L^2(\nu_n)}     \rvert,
\end{split}
\end{align*}
where in the last inequality we have used the properties of the curve $t \in [0,1] \mapsto (\tilde{\rho}_{n,t}, \tilde{V}_{n,t})$, Proposition \ref{AprioriBound}, and the fact that $\veps_n \diamP \rightarrow 0$ as $n \rightarrow \infty$. Now, using the fact that $\tilde \rho_0^h = J_h \ast \tilde{\rho}_0 $,
\begin{align*}
 \lvert \langle \tilde{\rho}_0^h, \phi \rangle_{L^2(\nu)}  -  \langle  \tilde{\rho}_{n,0}, J_h \ast \phi \rangle_{L^2(\nu_n)}     \rvert     &=  \lvert \langle \tilde{\rho}_0,  J_h \ast \phi \rangle_{L^2(\nu)}  -  \langle  \tilde{\rho}_{n,0}, J_h \ast \phi\rangle_{L^2(\nu_n)}   \rvert 
 \\& = \lvert  \int_{\T^d} (   J_h \ast \phi(x)  - J_h \ast \phi(T_n(x)) ) \tilde \rho_0(x) dx  \rvert
 \\& \leq C_3(h) \delta_n \lVert \phi \rVert_{L^1(\T^d)}   .
\end{align*}
Hence, for every $t\in [0,1]$,
\begin{equation}
 \lvert  \langle   \tilde{\rho}_t^h, \phi \rangle_{L^2(\nu)}  -  \langle  \tilde{ \rho}_{n,t}, J_h \ast \phi \rangle_{L^2(\nu_n)}   \rvert \leq C(h) \nc\lVert \phi \rVert_{L^1(\T^d)} (\veps_n + \frac{\delta_n}{\veps_n} ) . 
 \label{ineq1Tri}
\end{equation}
On the other hand, 
\begin{equation}
\lvert     \langle \tilde{\rho}_{n,t}, J_h \ast \phi \rangle_{L^2(\nu_n)} -   \langle  \tilde{\rho}_{t}, J_h \ast \phi  \rangle_{L^2(\nu)}  \rvert  = \left \lvert \int_{\T^d}(   J_h \ast \phi(x)- J_h\ast \phi(T_n(x))    ) \tilde{\rho}_t(x) dx \right \rvert  \leq \delta_n C_3(h)\lVert \phi \rVert_{L^1(\T^d)}.
\label{ineq2Tri}
\end{equation}
Finally,
\begin{align}
\begin{split}
\lvert   \langle  \tilde{\rho}_t, J_h \ast\phi \rangle_{L^2(\nu)} -  \langle  \tilde{\rho}_t, \phi \rangle_{L^2(\nu)}  \rvert & = \lvert    \langle   J_h \ast \tilde \rho_t, \phi \rangle_{L^2(\nu)} - \langle  \tilde \rho_t, \phi \rangle_{L^2(\nu)}   \rvert
\\& \leq \int_{T^d} \lvert \tilde \rho_t(x) - J_h \ast \tilde \rho_t(x)  \rvert\lvert \phi(x) \rvert dx 
\\& \leq \sum_{i,j} \int_{U_i}\int_{U_j}J_h(x-z ) \lvert \tilde \rho_{n,t}(\x_i) - \tilde \rho_{n,t}(\x_j) \rvert \lvert  \phi(x)\rvert dz dx
\\&\leq  C(b) \nc \sum_{i,j} \int_{U_i}\int_{U_j}J_h(x-z ) (\lvert \x_i - \x_j \rvert + \delta_n) \lvert  \phi(x)\rvert dz dx
\\& \leq C(b) \sum_{i,j} \int_{U_i}\int_{U_j}J_h(x-z ) (\lvert x-z \rvert  +  \delta_n) \rvert \lvert  \phi(x)\rvert dz dx 
\\& \leq  C(b)( \sqrt{h} + \delta_n ) \lVert \phi \rVert_{L^1(\T^d)}.
\end{split}
\label{ineq3Tri}
\end{align}

From the triangle inequality and \eqref{ineq1Tri}, \eqref{ineq2Tri}, \eqref{ineq3Tri}, it follows that
\begin{align*}
\begin{split}
\lvert \langle \tilde{\rho}_t^h - \tilde{\rho}_t , \phi \rangle_{L^2(\nu)} \rvert  & \leq C\left( C(h) (\veps_n + \frac{\delta_n}{\veps_n}) + C(b) (\sqrt{h}+ \delta_n)  \right)\lVert  \phi \rVert_{L^1(\T^d)}
\end{split}
\end{align*}
\nc
%
Since the previous inequality holds for arbitrary $\phi \in C^\infty(\T^d)$, we can conclude, via a density argument and the duality between $L^1(\T^d)$ and $L^\infty(\T^d)$, that
\begin{align}
\lVert \tilde{\rho}_t^h - \tilde{\rho}_t  \rVert_{L^\infty(\T^d)}  \leq C\left( C(h) (\veps_n + \frac{\delta_n}{\veps_n}) + C(b) (\sqrt{h}+ \delta_n)  \right) =: \xi(n,h,b).
\label{estimateforh}
\end{align}
We will later pick the parameters $h, b, a$ appropriately so that in particular 
\begin{equation}
  \xi(n,h,b ) \leq \frac{a}{3}. 
  \label{Main2}
\end{equation}
As we will see later on, this is possible since we send $n \rightarrow \infty$, $h \rightarrow 0$ and $b \rightarrow 0$, before we send $a \rightarrow 0$. Using \eqref{Main2} we see that both $\tilde{\rho}_t^h$ and $\tilde{\rho}_t$ are bounded below by $a /3$ for all $t \in [0,1]$. In particular, $\tilde{\rho}_t^h$ is a density and, moreover, from Lemma \ref{W2L8}
\begin{align}
\begin{split}
 W(\tilde\rho_{t}(x) dx, \tilde{\rho}_t^h(x) dx) & \leq C \left( C(s) + \frac{1}{a} \right)\xi(n,h,b) .
 \end{split}
 \label{AuxPart21003}
\end{align}

Let us now estimate the total action of $t \in [0,1] \mapsto (\tilde{\rho}_t^h, \vec{V}_t^h)$. As a consequence of \eqref{estimateforh} we have
\[  \tilde{\rho}_t^h(x) \geq (1- 2\frac{\xi(n,h,b)}{a}) \tilde \rho_t(x) \geq \frac{\tilde \rho_t(x)}{1+ C\xi(n,h,b) / a}   , \quad \forall x \in \T^d .\]
Also, from the definition of $\tilde{\rho}_t$ and the properties of $\tilde{\rho}_{n,t}$ we see that
\[ \tilde{\rho}_{t}(x) \geq  \frac{\tilde{\rho}_{t}(z)}{1 + \frac{C(b)}{a}\left(\lvert z-x \rvert + \delta_n \right) } , \quad \forall x, z \in \T^d. \]
Hence, 
\begin{align*}
\begin{split}
\int_{\T^d} \frac{\lvert \vec{V}_t^h(x) \rvert^2}{\tilde{\rho}_t^h(x)}dx  &   \leq (1+ C \frac{\xi(n,h,b)}{a} ) \int_{\T^d} \frac{\lvert \vec{V}_t^h(x)  \rvert^2 }{ \tilde \rho_t(x)}dx   
\\ &\leq (1+ C \frac{\xi(n,h,b)}{a} ) \int_{\T^d} \int_{\T^d} J_h(x-z) \frac{\lvert \vec{V}_t(z) \rvert^2}{\tilde \rho_t(x)}dzdx  
\\& \leq (1+ C \frac{\xi(n,h,b)}{a} )  \int_{\T^d} \int_{\T^d} J_h(x-z) ( 1+ \frac{C(b)}{a}(\lvert z-x \rvert +  \delta_n )  ) \frac{\lvert \vec{V}_t(z) \rvert^2}{ \tilde \rho_t(z)}dzdx
\\& \leq (1+ C\frac{\xi(n,h,b)}{a}  + C \frac{C(b)}{a} \sqrt{h} + C\frac{C(b)}{a}\delta_n ) \int_{\T^d}  \frac{\lvert \vec{V}_t(z) \rvert^2}{ \tilde \rho_t(z)}dz
\\& \leq (1+ C \frac{\xi(n,h,b)}{a} ) \int_{\T^d}  \frac{\lvert \vec{V}_t(z) \rvert^2}{\rho_t(z)}dz,
\end{split}
\end{align*}
where the second inequality follows from Jensen's inequality. Using Lemma \ref{LemmaAux2} with \[V(x,y):=\tilde V_{n,t}(T_n(x), T_n(y)) \frac{w_{\tilde{\veps}_n}(T_n(x), T_n(y))}{w_{\tilde\veps_n+2 \delta_n}(x,y)},\]  
and using the fact that 
\[ \tilde \rho_t(z)  \geq \frac{\theta( \tilde{\rho}_{n,t}(\x_i), \tilde{\rho}_{n,t}(\x_j) )}{1 + \frac{C(b)}{a}\veps_n  }, \quad  \forall z \in U_i, \quad \forall \x_j \text{ s.t } |\x_j - \x_i| \leq \tilde \veps _n, \]
we deduce that 
\begin{align*}
\begin{split}
\int_{\T^d}  \frac{\lvert \vec{V}_t(z) \rvert^2}{ \tilde \rho_t(z)}dz &\leq  \alpha_d \sigma_\eta (1+ C \frac{\delta_n}{\veps_n})  \int_{\T^d} \int_{\T^d}(\tilde V_{n,t}(T_n(z), T_n(y)))^2  \left(\frac{w_{\tilde \veps_n}(T_n(z), T_n(y)) }{w_{\tilde \veps_n + 2 \delta_n}(z,y)} \right)^2\frac{w_{\tilde \veps_n+ 2 \delta_n}(z,y )}{\tilde \rho_t(z)} dy dz
\\&   \leq   \alpha_d \sigma_\eta (1 + C \frac{\delta_n}{\veps_n})   \int_{\T^d} \int_{\T^d}(\tilde V_{n,t}(T_n(z), T_n(y)))^2 \frac{w_{\tilde \veps_n}(T_n(x),T_n(y) )}{\tilde \rho_t(z)} dy dz 
\\& \leq (1 + C \frac{\delta_n}{\veps_n})(1+ \frac{C(b)}{a} \veps_n  ) \alpha_d \sigma_\eta  \mathcal{A}_{n, \tilde{\veps}_n}(\tilde{\rho}_{n,t}, \tilde{V}_{n,t}),
\\&\leq (1 + C \frac{\delta_n}{\veps_n} + \frac{C(b)}{a}\veps_n) \alpha_d \sigma_\eta \mathcal{A}_{n, \tilde{\veps}_n}(\tilde{\rho}_{n,t}, \tilde{V}_{n,t})  
\end{split}
\end{align*}
Integrating the above inequality with respect to $t$ and using properties (1)-(5) of the curve $t \in [0,1] \mapsto (\tilde{\rho}_{n,t}, \tilde{V}_{n,t})$, we conclude that
\[ (W( \tilde{\rho}_{0}^h dx , \tilde{\rho}_1^h dx  ))^2\leq \alpha_d \sigma_\eta (W_n(\rho_{n,0}, \rho_{n,1}))^2 + C(\frac{\xi(n,h,b)}{a} + \frac{\delta_n}{\veps_n} + \frac{C(b)}{a}\veps_n ). \]

Combining the previous inequality with \eqref{AuxPart21003}, and using the triangle inequality, we deduce that
\begin{align}
\begin{split}
  W(\tilde{\rho}_0 dx, \tilde{\rho}_1 dx) & \leq \sqrt{\alpha_d \sigma_\eta} W_n(\rho_{n,0},\rho_{n,1})  + C\left(\frac{\xi(n,h,b)}{a} + \frac{\delta_n}{\veps_n} + \frac{C(b)}{a}\veps_n \right)^{1/2} + C(C(s) + \frac{1}{a}) \xi(n,  h,b).  
  \\& =: \sqrt{\alpha_d \sigma_\eta} W_n(\rho_{n,0},\rho_{n,1}) + \zeta(n,h,b,a,s).
\end{split}
\label{auxPart2}
\end{align}

Finally, from the triangle inequality, \eqref{auxPart2} and the properties of $\tilde{\rho}_{n,j}$ ($j=0,1$) we obtain
\begin{align}
\begin{split}
 &W(\mu_0, \mu_1) \leq W(\mu_0, \bH_s(\mu_0)) + W(\bH_s(\mu_0), P_n(\rho_{n,0})) + W(P_n(\rho_{n,0}), P_n(\tilde{\rho}_{n,0}) )  +  W( P_n(\tilde{\rho}_{n,0}), P_n(\tilde{\rho}_{n,1})  ) 
\\& + W(P_n(\tilde{\rho}_{n,1}), P_n(\rho_{n,1})  )  +   W( P_n(\rho_{n,1}), \bH_s(\mu_1)  ) + W(\bH_s(\mu_1), \mu_1)
 \\& \leq \sqrt{\alpha_d \sigma_\eta} W_n(\rho_{n,0}, \rho_{n,1} )  +   2\sqrt{s} + 2\delta_n + C(s)a   + C(s)(1+ \frac{1}{a})  ( b + \delta_n) + \zeta(n,h,b,a,s).
\end{split}
\label{Main3}
\end{align}


\subsection{Wrapping up}
Let $\rho_n \in \mathcal{P}(X_n)$ and for $s>0$ let $\rho^s: \T^d \rightarrow \R$ be the density  
\[ \rho^s:= \bH_s \circ P_n ( \rho_n ).\]
Then,
\begin{align}
\begin{split}
W_n( \rho_n, Q_n( \rho^s) ) &= W_n(Q_n( P_n(\rho_n)) , Q_n(\rho^s))  
\\ & \leq C W( Q_n( P_n(\rho_n))  , Q_n(\rho^s)) + C \veps_n\diamP
\\& \leq C\left(  W( Q_n( P_n(\rho_n)) , P_n(\rho_n) )  +  W( P_n(\rho_n), \rho^s ) + W(\rho^s , Q_n(\rho^s) )  \right) + \veps_n \diamP
\\& \leq C W( P_n(\rho_n), \rho^s ) + C \veps_n \diamP
\\&= C  W( P_n(\rho_n), \bH_s \circ P_n ( \rho_n ) ) + C\veps_n \diamP
\\& \leq C \sqrt{s} + C\veps_n \diamP
\end{split}
\label{AuxMainthm1}
\end{align}
where the first equality follows from the fact that $Q_n\circ P_n$ is the identity on $\mathcal{P}(X_n)$; the first inequality follows from Lemma \ref{AprioriBound}; the second inequality follows from the triangle inequality; the third inequality follows from the definition of the map $Q_n$ and the fact that the cells $U_i$ have diameter at most $2\delta_n$; the last inequality follows from Proposition \ref{HeatFlowTorus}. The above shows that for any given $\rho_n \in \mathcal{P}(X_n)$ we can find a measure $\mu \in \mathcal{P}(T_d)$ for which $ Q_n \circ \bH_s(\mu) $  is within distance  $C \sqrt{s} + C\veps_n \diamP$ from $\rho_n$ .


We can now take $n \rightarrow \infty$, so that $\veps_n \rightarrow 0$ , $\frac{\delta_n}{\veps_n} \rightarrow 0$ and $\veps_n \diamP \rightarrow 0$. Then we can take $h \rightarrow 0$, $b\rightarrow 0$, $a \rightarrow 0$ and finally $s \rightarrow 0$ (in that order) in all the above estimates to conclude that we can pick $h:=h_n$, $b:= b_n$, $a:=a_n$, $s:=s_n$ to guarantee that \eqref{ChoiceN1} and \eqref{Main2} are satisfied and that the quantities
\[ \varrho_n^1 := C \left(  \frac{\delta_n}{\veps_n} + \frac{\kappa(n,s_n)}{c_1(s_n)} + C(s_n)\veps_n \right)^{1/2} + C(s_n) \left( \veps_n +  \frac{\delta_n}{\veps_n}\right) + C \veps_n \diamP,   \]
\[  \varrho_n^2:= 2\sqrt{s_n} + 2\delta_n + C(s_n)a_n   + C(s_n)(1+ \frac{1}{a_n})  ( b_n + \delta_n) + \zeta(n,h_n,b_n,a_n,s_n), \]
\[  \varrho_n^3:=C \sqrt{s_n} + C\veps_n \diamP,\]
all converge to zero as $n \rightarrow \infty$. We can then define,
\[ \varrho_n:= \max\{\varrho_n^1, \varrho_n^2, \varrho_n^3 \},\]
and observe that from \eqref{Main1}, \eqref{Main3}, \eqref{AuxMainthm1} it follows that $F_n:=Q_n\circ \bH_{s_n} $ is a $\varrho_n$-isometry between $\mathcal{P}(\T^d, \frac{1}{\sqrt{\alpha _d \sigma_\eta}} W)$ and $(\mathcal{P}(X_n), W_n)$. We conclude that $(\mathcal{P}(X_n), W_n)$ converges in the Gromov-Hausdorff sense towards $(\mathcal{P}(\T^d), \frac{1}{\sqrt{\alpha_d \sigma_\eta}}W)$.

 { \textbf{Acknowledgements}
	The author would like to thank Dejan Slep\v{c}ev for enlightening discussions and for introducing him to the line of research investigated in this work. The author would also like to thank Jan Maas for enlightening discussions on this and other related topics. This manuscript was completed while the author was visiting the Erwin Schr\"{o}dinger Institute to participate in the workshop ``Optimal Transport: from Geometry to Numerics". The author wants to thank the Institute for hospitality.}

%
%
%
%
%

\bibliography{Biblio}
\bibliographystyle{siam}

\end{document}